\documentclass[twocolumn]{autart}    

\usepackage[draft]{hyperref}
\usepackage{amsfonts}
\usepackage{amsmath}
\usepackage{graphicx}
\usepackage{algorithm}
\usepackage{algpseudocode}
\usepackage{placeins}
\usepackage{tikz}
\usepackage{tabularx}
\usepackage{array}
\usepackage[authoryear]{natbib}

\newtheorem{proposition}{\textbf{Proposition}}
\newtheorem{lemma}{\textbf{Lemma}}
\newtheorem{assumption}{\textbf{Assumption}}
\newtheorem{theorem}{\textbf{Theorem}}
\newtheorem{corollary}{\textbf{Corollary}}
\newtheorem{definition}{\textbf{Definition}}
\newtheorem{example}{Example}
\newtheorem{remark}{Remark}

\begin{document}

\begin{frontmatter}

\title{Analysis of Stability and Performance of Economic Model Predictive Control with State-Independent Costs} 


\author[Electrical]{Alireza Arastou}\ead{aarastou@student.unimelb.edu.au},    
\author[Math]{Ye Wang}\ead{ye.wang@unimelb.edu.au},               
\author[Electrical]{Erik Weyer}\ead{ewey@unimelb.edu.au}  

\address[Electrical]{Department of Electrical and Electronic Engineering, University of Melbourne, VIC 3010, Australia}  
\address[Math]{School of Mathematics and Statistics, University of Melbourne, VIC 3010, Australia}             
          

\begin{abstract}                          
This paper studies economic model predictive Control (EMPC) schemes, where the stage cost depends only on control inputs. Such problems arise in applications like water distribution networks and differ from standard EMPC since multiple steady states can correspond to the unique optimal steady input. We show that, under a strict dissipativity assumption related to the set of optimal steady states, the closed-loop trajectories converge asymptotically to this set, ensuring convergence of the economic cost to the optimal steady state cost. To enhance Lyapunov stability, we propose a modified stage cost that preserves the optimal input while guaranteeing asymptotic stability of a specific equilibrium with a slight performance loss. The approach is further extended to EMPC of a class of linear systems with periodic costs and disturbances by lifting it to a multi-step EMPC problem for periodic operations. A case study with a water distribution network demonstrates the effectiveness of the proposed methods in achieving both asymptotic convergence and stability.
\end{abstract}

\end{frontmatter}

\section{Introduction} \label{sec: Introduction}
Economic model predictive control (EMPC) is a type of MPC prioritising economic objectives over tracking references. EMPC enables controllers to directly optimise system operation, making it particularly suitable for a variety of industrial applications, such as chemical processes \citep{ellis2017economic}, energy systems \citep{hu2023economic}, water distribution networks (WDNs) \citep{wang2016periodic}, and traffic systems \citep{he2024dynamic}.

A key concept in EMPC is the optimal steady state problem. It was shown in \citep{angeli2011average,amrit2011economic,rawlings2017model} that the optimal steady state problem provides a pair of optimal state and input trajectory such that the economic cost is minimised and the system is in the steady state. In many applications, such as pumping energy minimisation in WDNs \citep{arastou2025optimization,ocampo2012hierarchical}, the stage cost depends only on the control input, i.e., $\ell(u)$. Assuming strict convexity of $\ell(u)$, the optimal input from steady state problem, denoted by $u^s$ is unique; however, multiple or an infinite number of states can satisfy steady state system dynamics with $u^s$ as the input. Consequently, it is desirable to show asymptotic convergence of the closed-loop trajectories given by EMPC scheme to the set of steady states, not a unique point. Existing EMPC formulations that rely on terminal sets or stability notions centred on a single steady state \citep{amrit2011economic,Diehl2011,angeli2011average} are therefore not directly applicable, motivating the need for alternative approaches.

Motivated by the above challenge, this paper uses strict dissipativity with respect to the set of steady states to analyse EMPC closed-loop behaviour. The main contributions are:

\begin{itemize}
    \item Theoretical guarantees on asymptotic convergence into the set of optimal steady states under a strict dissipativity assumption;
    
    \vspace{0.1 cm}
    
    \item A method for showing asymptotic stability of a specific equilibrium point by modifying the stage cost; 
    
    \vspace{0.1 cm}
    
    \item Extension of the proposed method to linear systems with periodic cost and disturbances.
\end{itemize}

These contributions build upon and extend the existing EMPC literature. The foundational work \citep{angeli2011average} established that EMPC ensures asymptotic average performance no worse than the optimal steady state solution. A Lyapunov function was proposed in \citep{Diehl2011} for linear systems with an EMPC-based control system under a strong duality assumption. This was generalised in \citep{angeli2011average}, where they showed that under a strict dissipativity assumption, the solution of the optimal steady state problem is an asymptotically stable equilibrium point for the system under the EMPC control. Our first contribution generalises this viewpoint by shifting the stability target from a single steady state to a set of optimal steady states.

Dissipativity has since become a critical part of EMPC theory. A comprehensive survey of dissipativity, emphasising its central role in linking EMPC stability, turnpike behaviour, and performance guarantees, was given by \citep{Grune_Dissipativity}. The concept of turnpikes, originating in optimal control, refers to the property that optimal trajectories spend most of their time near steady state or periodic optimal regimes. The relation between dissipativity and turnpike properties in EMPC was established in \citep{grune2016relation}. Moreover, \citep{GRUNE20141187} presented conditions on strict dissipativity and controllability such that the closed-loop system is practically asymptotically stable without terminal constraints. A new dissipativity condition for proving asymptotic stability of discounted EMPC was given in \citep{zanon2022new}. The presented dissipativity condition took into account the discount factor, and a new Lyapunov function was defined correspondingly. Our results make use of a modified stage cost and a strict dissipativity condition with respect to the set of steady states to guarantee asymptotic convergence to the optimal steady state cost and asymptotic stability. 

Alternative approaches using average constraints were introduced in \citep{angeli2011average} to ensure a desired average performance. The average constraints were also used in \citep{MULLER_Average} to enforce asymptotic convergence of the closed-loop system to the optimal steady states in the case that dissipativity assumption is not met. While the aforementioned average constraints guarantee the desired average performance asymptotically, the authors in \citep{MULLER_Transient_Average} used transient average constraints to ensure the desired performance within a finite time. Nevertheless, the use of average constraints to enforce asymptotic convergence can deteriorate the transient performance considerably \citep{MULLER_Average} and can result in infeasibility \citep{MULLER_Transient_Average}. In contrast, this work relies on a dissipativity-based condition that avoids imposing explicit average constraints while still guaranteeing convergence properties.

Research on periodic operation has become increasingly important, reflecting real-world applications where economically optimal trajectories are inherently time-varying. An EMPC formulation with a periodic terminal constraint was introduced in \citep{angeli2011average}, proving that the closed-loop average performance is no worse than that of the optimal periodic orbit. Dissipativity notions in periodic EMPC were presented in \citep{kohler2018periodic, zanon2016periodic}. The periodic EMPC strategy was extended in \citep{muller2016economic} to a multi-step EMPC without terminal constraints, showing convergence to periodic operation under strict dissipativity assumption. In \citep{kohler2020periodic} an artificial reference was used in an EMPC scheme to guarantee recursive feasibility and performance bounds for periodic operation. This method avoided the need for precomputing periodic trajectories. Furthermore, \citep{wang2016periodic} developed periodicity-constrained EMPC for linear systems, optimising over all periodic trajectories through the current state, thereby ensuring recursive feasibility and convergence in convex settings. Our third contribution connects this literature with dissipativity-based convergence to steady state sets, showing how periodic disturbances and costs can be naturally incorporated.

The proposed method is applied to a case study with a WDN, and the asymptotic convergence and stability are shown under suitable conditions. The remainder of the paper is organised as follows: The problem statement is given in Section \ref{sec: problem statement}. Asymptotic convergence and stability analysis are presented in Section \ref{sec: performance and stability analysis}. Section \ref{sec: application to periodic} gives the results for linear systems with periodic stage cost and disturbances. Simulation results for a WDN are given in Section \ref{sec: simulation results}, and Section \ref{sec: conclusion} gives conclusions.

\section{Problem Statement} \label{sec: problem statement}
This section gives the system dynamics, control objective, and constraints to be used in the EMPC scheme.
We consider discrete time nonlinear systems in the form of
\begin{equation} \label{eq: Nonlinear Model dynamics}
x_{t+1}=f(x_t,u_t),
\end{equation}
where $f(\cdot)$ represents the system dynamics, $x_t \in \mathcal{X} \subseteq \mathbb{R}^{n}$ denotes states, and $u_t \in \mathcal{U} \subseteq \mathbb{R}^{m}$ denotes control inputs at time $t$. $\mathcal{X}$ and $\mathcal{U}$ are constraint sets on states and inputs. 
\begin{assumption} \label{ass: continuity of model}
    The system dynamics $f(\cdot,\cdot)$ in \eqref{eq: Nonlinear Model dynamics} is continuous on $\mathcal{X}\times \mathcal{U}$.
\end{assumption}
\begin{assumption} \label{ass: compact and convex constraint}
    $\mathcal{X}$ and $\mathcal{U}$ are compact and convex.
\end{assumption}

Motivated by WDNs, a pure economic control objective is considered where the stage cost is a function of control inputs only. It is denoted by $\ell(u): \mathcal{U} \rightarrow \mathbb{R}$. 
\begin{assumption} \label{Ass: Strict convexity of l(u)}
    The stage cost $\ell(\cdot)$ is continuous and strictly convex.
\end{assumption}

Next, the set of optimal steady-state trajectories can be obtained by solving the following optimisation problem:
\begin{subequations} \label{eq: SS equations}
    \begin{align}
        & \min_{x,u} \ell(u), \label{eq: SS Objective function}\\
        \text{subject to:}& \nonumber \\
        & x=f(x,u), \label{eq: model dynamics SS}\\
        & x \in \mathcal{X}, \hspace{0.5 cm} u \in \mathcal{U}.\label{eq: state and input constraints SS}
    \end{align}
\end{subequations}
\begin{assumption} \label{ass: existance of steady state solution}
    At least one solution $(x,u)$ to the steady-state problem in \eqref{eq: SS equations} exists.
\end{assumption}
Denote the optimal cost in steady-state given by \eqref{eq: SS equations} by $\ell^s$. 
\begin{remark}
The system dynamics in \eqref{eq: Nonlinear Model dynamics} and \eqref{eq: model dynamics SS} are nonconvex in general. Hence, the solution of \eqref{eq: SS equations} might not be unique. If the system dynamics is convex, then the optimal input, denoted by $u^s$, is unique. In this case, any state meeting
\begin{equation}
    x=f(x,u^s), \label{eq: steady state model}
\end{equation} 
is an optimal steady state as it gives the optimal steady state cost $\ell^s$. In general, \eqref{eq: steady state model} can have a unique, multiple or infinite number of solutions.    
\end{remark}

The set of optimal steady state solutions is characterised by
\begin{equation} \label{eq: Steady state set}
    \mathcal{Z}^s=\{(x,u) \in \mathcal{X} \times \mathcal{U}\hspace{0.1 cm} | \hspace{0.1 cm}x=f(x,u), \hspace{0.1 cm} \ell(u)=\ell^s\},
\end{equation}
Also, denote the projection of $\mathcal{Z}^s$ onto $\mathcal{X}$ by $\mathcal{X}^s$, i.e.,
\begin{equation} \label{eq: definition of Xs}
    \mathcal{X}^s=\{x \in \mathcal{X}\hspace{0.1 cm} | \hspace{0.1 cm} \exists u \in \mathcal{U}, \hspace{0.1 cm} \text{s.t.} \hspace{0.1 cm} (x,u) \in \mathcal{Z}^s \}.
\end{equation}
The EMPC problem at time $t$ is given by
\begin{subequations} \label{eq: EMPC Formulation}
    \begin{align}
         \min_{u_{0|t},\ldots,u_{N-1|t}} &\sum_{k=0}^{N-1} \ell(u_{k|t}), \label{eq: objective function in control problem}\\
        \text{subject to:} \hspace{0.3 cm} &k=0,1,\ldots,N-1 \nonumber\\
        & x_{0|t}=x_t, \label{eq: initialization in EMPC}\\
        & x_{k+1|t}=f(x_{k|t},u_{k|t}), \label{eq: Model dynamics in EMPC}\\
        & x_{k|t} \in \mathcal{X},\label{eq: State constraints}\\
        &u_{k|t} \in \mathcal{U},\label{eq: Input constraints}\\
        & x_{N|t} \in \mathcal{X}^s. \label{eq: Terminal constraint}
    \end{align}
\end{subequations}
The formulation in \eqref{eq: EMPC Formulation} requires that the terminal state belongs to the optimal steady-state set $\mathcal{X}^s$. Denote the solution to \eqref{eq: EMPC Formulation} by
\[\mathbf{u}_t=[u_{0|t}^\top,\ldots,u_{N-1|t}^\top]^\top.\]
The admissible set of $(x,\mathbf{u})$ pair at time $t=0$ is 
\begin{equation} \label{eq: Admissble pair}
\begin{split}
&\mathcal{Z}_{N}=\big\{(x,\mathbf{u}) \in \mathcal{X}\times \mathcal{U}^N \hspace{0.1 cm}|\hspace{0.1 cm} \exists \hspace{0.1 cm} x_{k}\in\mathcal{X},\hspace{0.1 cm} k=1,\ldots,N-1\\ &\text{such that}\hspace{ 0.1 cm}x_{0}=x, \hspace{0.1 cm} x_{k+1}=f(x_{k},u_{k}),
\hspace{0.1 cm} x_{N}\in \mathcal{X}^s\big\}, 
\end{split}
\end{equation}
and the set of admissible initial states is
\begin{equation} \label{eq: admissible initial states}
    \mathcal{X}_{N}=\{x \in \mathcal{X} \hspace{0.1 cm}| \hspace{0.1 cm} \exists \mathbf{u} \hspace{0.1 cm} \text{such that} \hspace{0.1 cm} (x,\mathbf{u})\in\mathcal{Z}_{N}\}.
\end{equation}
It is desirable to steer the system in \eqref{eq: Nonlinear Model dynamics} to an equilibrium in $\mathcal{Z}^s$ using the EMPC strategy. This way, it is guaranteed that the stage cost $\ell(u)$ will also converge to $\ell^s$. Note that we are not requiring convergence to a specific point within $\mathcal{Z}^s$. The following example highlights the difference between the problem in \eqref{eq: EMPC Formulation} and traditional EMPC schemes. 

\begin{example} \label{ex: System with integrators}
    Consider an integrator dynamic system $x_{k+1}=x_k+u_k$. Let $\mathcal{X}=\{x\hspace{0.1 cm}|\hspace{0.1 cm}-1\leq x \leq 1\}$, $\mathcal{U}=\{u\hspace{0.1 cm}|\hspace{0.1 cm}-1\leq u \leq 1\}$ and $\ell(u)=u^2$. By solving the problem \eqref{eq: SS equations}, the optimal input $u^s=0$ is obtained; however, any feasible $x \in \mathcal{X}$ is an equilibrium. The main goal is to converge to the minimum steady state cost $\ell^s=0$. This is different from traditional EMPC schemes, where states are steered toward a unique steady state $x^s$.  
\end{example}  

We now provide a rigorous analysis of the closed-loop properties of the scheme. Section \ref{sec: performance and stability analysis} develops conditions ensuring recursive feasibility, asymptotic convergence of the closed-loop trajectories, and stability guarantees.

\section{Asymptotic Performance and Stability Analysis} \label{sec: performance and stability analysis}

This section is divided into two parts. First, we prove
recursive feasibility of the EMPC and show that the average economic cost in the long run is at least as good as $\ell^s$. Second, we establish asymptotic convergence to the set of optimal steady states $\mathcal{X}^s$ and demonstrate asymptotic stability via a Lyapunov-type argument using a modified stage cost.

\subsection{Recursive feasibility and performance analysis}

This subsection establishes recursive feasibility and provides asymptotic average performance guarantees of the closed-loop system. The detailed statements and proofs are given in the following propositions. 

\begin{proposition} \label{Prop: Recursive feasibility centralised}
    Suppose that Assumption \ref{ass: existance of steady state solution} holds. Given a feasible solution at time $t_0$, the EMPC problem in \eqref{eq: EMPC Formulation} is recursively feasible for all $t\geq t_0$.  
\end{proposition}

\noindent \textbf{Proof}. Denote a solution to \eqref{eq: EMPC Formulation} at time $t=t_0$ and the corresponding states by 
\begin{equation} \label{eq: solution at time t0}
    \left\{x^0_{0|t}, \hspace{0.1 cm} x^0_{1|t}, \hspace{0.1 cm} \ldots \hspace{0.1 cm} x^0_{N|t}\right\}, \hspace{0.1 cm} \left\{u^0_{0|t}, \hspace{0.1 cm} \ldots \hspace{0.1 cm} u^0_{N-1|t}\right\} 
\end{equation}
An admissible pair of states and inputs at time $t+1$ is obtained by shifting \eqref{eq: solution at time t0} as given below
\begin{equation} \label{eq: Shifted solution}
\begin{split}
    &\left\{x^0_{1|t}, \hspace{0.1 cm} \ldots, \hspace{0.1 cm} x^0_{N|t},\hspace{0.1 cm} x^0_{N|t}\right\},\\
    &\left\{u^0_{1|t}, \hspace{0.1 cm} \ldots, \hspace{0.1 cm} u^0_{N-1|t}, \hspace{0.1 cm} u^s\right\},
\end{split}
\end{equation}
where $u^s$ is an optimal input such that $(x^0_{N|t},u^s) \in \mathcal{Z}^s$. Using the shifted sequences, constraints \eqref{eq: Model dynamics in EMPC}-\eqref{eq: Input constraints} are satisfied with the optimal solution at time $t$. From \eqref{eq: Steady state set}, it follows that $x^0_{N|t} \in \mathcal{X}^s \subseteq \mathcal{X}$ remains unchanged if $u^s$ is applied to the system. Thus, the terminal state also meets \eqref{eq: Terminal constraint}. Therefore, the EMPC problem in \eqref{eq: EMPC Formulation} is recursively feasible at any time $t\geq t_0$ by induction. $\square$

\begin{proposition}
    Let Assumptions \ref{ass: continuity of model}-\ref{Ass: Strict convexity of l(u)} hold. The open-loop cost of the EMPC problem in \eqref{eq: EMPC Formulation} is no worse than the cost obtained when enforcing an equality terminal constraint at a given fixed steady state, i.e., $x_{N|t}=x^s \in \mathcal{X}^s$. Moreover, the system in \eqref{eq: Nonlinear Model dynamics} in closed-loop with the control generated by the EMPC scheme in \eqref{eq: EMPC Formulation} has an asymptotic average performance no worse than $\ell^s$, i.e.,
    \[
    \limsup_{T\rightarrow \infty}{\frac{\displaystyle\sum_{t=0}^{T} \ell(u_{0|t}^0)}{T+1}} \leq \ell^s.
    \]
\end{proposition}
\noindent \textbf{Proof}. The admissible pair of $(x,\mathbf{u}_t)$ for the EMPC problem with the equality constraint $x_{N|t}=x^s \in \mathcal{X}^s$ is $\subseteq \mathcal{Z}_N$ in \eqref{eq: Admissble pair}, since both problems share the same model \eqref{eq: Model dynamics in EMPC}, state, and input constraints in \eqref{eq: State constraints}-\eqref{eq: Input constraints} and $x^s \in \mathcal{X}^s$. Thus, the solution of the optimisation problem in \eqref{eq: EMPC Formulation} is no worse than the solution with a single point as the terminal set. Moreover, denote the optimal value of the cost in \eqref{eq: objective function in control problem} at time $t$ by 
\[J(\mathbf{u}_t^0)=\sum_{k=0}^{N-1} \ell(u^0_{k|t}).\] 
Using the shifted sequences denoted by $\hat{\mathbf{u}}_{t+1}$ in \eqref{eq: Shifted solution} we have
\begin{equation}
        J(\mathbf{u}_{t+1}^0)-J(\mathbf{u}_{t}^0) \leq J(\hat{\mathbf{u}}_{t+1})-J(\mathbf{u}_{t}^0) \nonumber = \ell^s-\ell(u^0_{0|t}).
\end{equation}
Taking asymptotic average on both sides gives
\begin{equation} \label{eq: liminf inequality}
\begin{split}
    \liminf_{T\rightarrow \infty}&{\frac{\sum_{t=0}^{T} J(\mathbf{u}_{t+1}^0)-J(\mathbf{u}_{t}^0)}{T+1}}\\
    &\leq \ell^s-\limsup_{T\rightarrow \infty}{\frac{\sum_{t=0}^{T} \ell(u_{0|t}^0)}{T+1}}.
\end{split}
\end{equation}
According to Assumptions \ref{ass: continuity of model}-\ref{ass: compact and convex constraint} (continuity of $\ell(\cdot)$ and compactness of $\mathcal{U}^N$), $J(\cdot)$ is bounded by extreme value theorem, that is $J(\cdot)$ must attain a maximum and a minimum within $\mathcal{U}^N$; thus, 
\begin{equation}
    \liminf_{T\rightarrow \infty}{\frac{J(\mathbf{u}_{T+1}^0)-J(\mathbf{u}_{0}^0)}{T+1}}=0. \nonumber
\end{equation}
Hence, it can be deduced from \eqref{eq: liminf inequality} that
\[
\limsup_{T\rightarrow \infty}{\frac{\displaystyle\sum_{t=0}^{T} \ell(u_{0|t}^0)}{T+1}} \leq \ell^s,
\]
which gives the asymptotic average performance. $\square$


\subsection{Asymptotic convergence to the optimal cost in steady state} \label{sec: Asymptotic convergence}

Asymptotic convergence to $\ell^s$ and asymptotic stability analysis using a modified stage cost are given in this section. We first recall the definitions for asymptotic convergence and stability.

\begin{definition} \label{def: asymptotic convergence}
    The distance between a point $\varphi \in \mathbb{R}^n$ and a nonempty compact set $\mathcal{A}$ is denoted by $||\varphi||_{\mathcal{A}}= \min_{z\in \mathcal{A}} ||\varphi-z||_2$. A sequence $\varphi_t$ asymptotically converges to a set $\mathcal{A}$ if $\lim_{t \rightarrow \infty} ||\varphi_t||_{\mathcal{A}}=0$. 
\end{definition}
\begin{definition} (\citep{Khalil:1173048})
    A continuous function $\alpha: [0, a) \rightarrow [0, \infty)$ belongs to class $\mathcal{K}$ if $\alpha(0)=0$ and it is strictly increasing. 
\end{definition}
\begin{definition}(\citep{Khalil:1173048,rawlings2017model}) \label{def: Lyapunov function definition}
Let $\mathcal{A}\subseteq \bar{\mathcal{X}} \subset \mathbb{R}^n$ be a closed set of equilibria of the system $x_{t+1}=f(x_t)$. A continuous function $V : \bar{\mathcal{X}} \to \mathbb{R}_{\geq 0}$ is called a Lyapunov function on $\bar{\mathcal{X}}$ 
if there exist functions $\alpha_1,\alpha_2,\in \mathcal{K}$ 
and a continuous positive definite function $\alpha_3$ such that for all $x \in \bar{\mathcal{X}}$
\begin{equation}
    \begin{split}
        & \alpha_1(\|x\|_{\mathcal{A}}) \;\le\; V(x) \;\le\; \alpha_2(\|x\|_\mathcal{A}),\\
        & V(f(x)) - V(x) \;\le\; -\alpha_3(\|x\|_{\mathcal{A}}). \nonumber
    \end{split}
\end{equation}
\end{definition}

\begin{lemma}(\citep{Khalil:1173048}) \label{lem: AS theorem from Khalil}
Suppose $\bar{\mathcal{X}} \subset \mathbb R^n$ is positively invariant for the system $x_{t+1} = f(x_t)$, i.e., if $x_0 \in \bar{\mathcal{X}}$ then $x_t \in \bar{\mathcal{X}},$ $ \forall t \in \mathbb{N}$. Let $\mathcal A \subseteq \bar{\mathcal{X}}$ be a closed set of equilibria, i.e.\ 
$f(a)=a$ for all $a\in\mathcal A$. If there exists a Lyapunov function in $\bar{\mathcal{X}}$ with respect to $\mathcal A$, 
then $\mathcal A$ is asymptotically stable in $\bar{\mathcal{X}}$.
\end{lemma}
\noindent \textbf{Proof}. See the proof of Theorem 4.9 in \citep{Khalil:1173048}. $\square$ 

Consider the set of optimal steady state solutions in \eqref{eq: Steady state set}. The following strict dissipativity assumption with respect to $\mathcal{Z}^s$ is imposed \citep{zanon2016periodic}.
\begin{assumption} \label{Ass: Strict dissipativity}
    The system in \eqref{eq: Nonlinear Model dynamics} is strictly dissipative with respect to the supply rate $\ell(u)-\ell^s$, i.e.,  there exists a storage function $\lambda(x): \mathcal{X}\rightarrow \mathbb{R}$ and a positive definite $\rho(||(x,u)||_{\mathcal{Z}^s}):\mathcal{X}\times \mathcal{U} \rightarrow \mathbb{R}_{\geq 0}$  such that
\begin{equation} \label{eq: strict dissipativity def with respect to Xs}
    \lambda(x)-\lambda(f(x,u))+\ell(u)-\ell^s \geq \rho(||(x,u)||_{\mathcal{Z}^s}),
\end{equation}
for all $(x,u)\in\mathcal{X}\times \mathcal{U}$.
\end{assumption}

The rotated stage and terminal costs are defined as follows \citep{angeli2011average}:
\begin{subequations} \label{eq: rotated costs}
\begin{align}
&L(x,u)=\ell(u)-\ell^s+\lambda(x)-\lambda(f(x,u)),\\
&V^f_r(x)=\lambda(x). 
\end{align}
\end{subequations}
Using \eqref{eq: rotated costs}, consider the following auxiliary EMPC problem:
\begin{subequations} \label{eq: Rotated EMPC Formulation}
    \begin{align}
        & \min_{u_{0|t},\ldots,u_{N-1|t}} \sum_{k=0}^{N-1} L(x_{k|t},u_{k|t})+V^f_r(x_{N|t}), \label{eq: Rotated objective function in control problem}\\
        &\text{subject to:} \hspace{0.3 cm} k=0,1,\ldots,N-1 \nonumber\\
        & x_{0|t}=x_t, \label{eq: Rotated initialization in EMPC}\\
        & x_{k+1|t}=f(x_{k|t},u_{k|t}), \label{eq: Rotated model dynamics in EMPC}\\
        & x_{k|t} \in \mathcal{X},\label{eq: Rotated state constraints}\\
        &u_{k|t} \in \mathcal{U},\label{eq: Rotated input constraints}\\
        & x_{N|t} \in \mathcal{X}^s \label{eq: Rotated terminal constraint}
    \end{align}
\end{subequations}
The following lemma guarantees equivalence of solutions to EMPC problems in \eqref{eq: EMPC Formulation} and \eqref{eq: Rotated EMPC Formulation}.

\begin{lemma}
    Under Assumptions \ref{Ass: Strict convexity of l(u)} and \ref{Ass: Strict dissipativity}, the optimal solutions to the EMPC problems in \eqref{eq: EMPC Formulation} and \eqref{eq: Rotated EMPC Formulation} are identical. 
\end{lemma}
\noindent\textbf{Proof}. From \eqref{eq: rotated costs}, the objective function in \eqref{eq: Rotated objective function in control problem} can be simplified as follows:
    \begin{equation}
    \begin{split}
        & \sum_{k=0}^{N-1} L(x_{k|t},u_{k|t})+V^f_r(x_{N|t})\\
        &=\sum_{k=0}^{N-1} \ell(u_{k|t})+\lambda(x_{0|t})-\lambda(x_{N|t})-N\ell^s+V^f_r(x_{N|t})\\
        &=\sum_{k=0}^{N-1} \ell(u_{k|t})+\lambda(x_{0|t})-N\ell^s, \nonumber
    \end{split}
    \end{equation}
    where $x_{0|t}$ is the given initial state at time $t$. It shows that the difference between objective functions in \eqref{eq: objective function in control problem} and \eqref{eq: Rotated objective function in control problem}, i.e., $\lambda(x_{0|t})-N\ell^s$ is a constant term. Thus, the optimisation problems give identical solutions.  $\square$

We are now ready to prove asymptotic convergence to the set of optimal steady states set using strict dissipativity in Assumption \ref{Ass: Strict dissipativity} and recursive feasibility in Proposition \ref{Prop: Recursive feasibility centralised}.

\begin{theorem} \label{thm: Asymptotic convergence}
    Let Assumptions \ref{ass: continuity of model}-\ref{Ass: Strict dissipativity} hold. Consider the system \eqref{eq: Nonlinear Model dynamics} under the EMPC scheme in \eqref{eq: Rotated EMPC Formulation}. The closed-loop solutions, i.e., $(x^0_{0|t},u^0_{0|t})$, converge to $\mathcal{Z}^s$ asymptotically. Consequently, $\ell(u_t)\rightarrow \ell^s$ as $t\rightarrow \infty$.
\end{theorem}

\noindent \textbf{Proof}. The optimal cost in \eqref{eq: Rotated objective function in control problem} is used as a Lyapunov-like function
    \begin{equation} \label{eq: Lyapunov-like function}
        V^0(x_t)= \sum_{k=0}^{N-1} L(x^0_{k|t},u^0_{k|t})+V^f_r(x^0_{N|t}), 
    \end{equation}
where the index `$0$' corresponds to the optimal solutions of \eqref{eq: Rotated EMPC Formulation} and $x_t$ is the initial state at time $t$. From Proposition \ref{Prop: Recursive feasibility centralised}, a candidate feasible solution at time $t+1$ is obtained by shifting the optimal state and input trajectories at time $t$. Denote the evaluated cost using the shifted sequence by $\hat{V}(\cdot)$. By optimality and from \eqref{eq: strict dissipativity def with respect to Xs}, we have
\begin{equation} \label{eq: descent property for Lyp-like in EMPC 1}
\begin{split}
    V^0(x_{t+1})- &V^0(x_t)\leq \hat{V}(x_{t+1})- V^0(x_t)\\
    &\leq L(x^0_{N|t},u^s(x^0_{N|t}))-L(x_t,u_{0|t}^0)\\
    &=-L(x_t,u_{0|t}^0)\leq -\rho(||(x_t,u_{0|t}^0)||_{\mathcal{Z}^s}).
\end{split}
\end{equation}
From \eqref{eq: descent property for Lyp-like in EMPC 1}, $V^0(x)$ is monotonically decreasing as long as $(x_t,u_{0|t}^0)$ is not in the set $\mathcal{Z}^s$. Moreover, $V^0(x_t)$ is continuous and bounded since $\mathcal{X}$ is compact. Hence, $V^0(x_t)$ converges to an equilibrium as $t \rightarrow \infty$ and $V^0(x_{t+1})- V^0(x_t) \rightarrow 0$. Consequently, $||(x_t,u_{0|t}^0)||_{\mathcal{Z}^s} \rightarrow 0$, i.e., $||(x_t,u_{0|t}^0)||\rightarrow \mathcal{Z}^s$ asymptotically. $\square$

\begin{remark}
    Theorem \ref{thm: Asymptotic convergence} does not guarantee that $\mathcal{X}^s$ is a set of asymptotically stable equilibrium points for the system in \eqref{eq: Nonlinear Model dynamics}; as there is no $\rho{(x)} \in \mathcal{K}$ such that $V^0(x) \leq \rho_1(||x||_{\mathcal{X}^s})$ as required in Definition \ref{def: Lyapunov function definition}. Thus, Lyapunov stability cannot be established using \eqref{eq: Lyapunov-like function} as a Lyapunov function. States starting sufficiently close to $\mathcal{X}^s$ are allowed to get away from it and asymptotically converge to $\mathcal{X}^s$. Nevertheless, this is not restrictive in the considered EMPC problem with a pure economic objective in \eqref{eq: EMPC Formulation}. From a practical point of view, keeping states within acceptable bounds and asymptotic convergence to the optimal cost in steady state is the main goal. From the recursive feasibility of the EMPC problem in Proposition \ref{Prop: Recursive feasibility centralised}, it is evident that closed-loop states are always bounded due to state constraints. Therefore, asymptotic convergence to $\ell^s$ is sufficient to achieve the economic goal using the control strategy in \eqref{eq: Rotated EMPC Formulation}. 
\end{remark}

Although the presented result is desirable in many applications, one might require asymptotic stability in addition to convergence to the optimal cost in steady state. We will address this challenge by modifying the stage cost in the next subsection. 

\subsection{Asymptotic stability by modifying stage cost} \label{sec: asymptotic stability of modified xs}

The main challenge in Section \ref{sec: Asymptotic convergence} was that the upper bounding class $\mathcal{K}$ function did not exist for the Lyapunov candidate function in \eqref{eq: Lyapunov-like function}. We will address this challenge by modifying the economic stage cost such that the difference between the optimal cost in steady state for the modified problem and the original problem in \eqref{eq: SS equations} is less than a small positive value denoted by $\gamma$. To this end, the modified steady state problem is formulated as
\begin{subequations} \label{eq: Modified SS equations}
    \begin{align}
        & \min_{x,u} \ell(u)+\varepsilon x^\top x, \label{eq: Modified SS Objective function}\\
        \text{subject to}& \nonumber \\
        & x=f(x,u), \label{eq: Modified model dynamics SS}\\
        & x \in \mathcal{X}, \hspace{0.5 cm} u \in \mathcal{U}, \label{eq: Modified state and input constraints SS}
    \end{align}
\end{subequations}
where $\varepsilon>0$. 
\begin{assumption} \label{ass: Uniqueness of steady state assumption}
    The solution to \eqref{eq: Modified SS equations} denoted by $(x_{\varepsilon}^s,u_{\varepsilon}^s)$ is unique.
\end{assumption}
\begin{lemma} \label{lem: bound for epsilon}
    Let Assumptions \ref{Ass: Strict convexity of l(u)} and \ref{ass: Uniqueness of steady state assumption} hold. For any $\gamma>0$, if $\varepsilon$ in \eqref{eq: Modified SS equations} is chosen as $\min\left\{\gamma, \hspace{0.1 cm}\frac{\gamma}{R}\right\}$, where $R=\displaystyle \max_{x \in \mathcal{X}} x^\top x$, then
    \begin{equation} \label{eq: modifed cost difference}
        \ell(u_\varepsilon^s)-\ell^s \leq \gamma.
    \end{equation} 
\end{lemma}
\noindent \textbf{Proof}. From optimality and the fact that feasibility region of the \eqref{eq: SS equations} and \eqref{eq: Modified SS equations} are the same, it holds that
\begin{equation} \label{eq: cost difference l(u_s) and modified one}
\begin{split}
    &  \ell(u^s_\varepsilon) \leq \ell(u^s_\varepsilon)+\varepsilon (x^s_\varepsilon)^\top x^s_\varepsilon \leq \ell^s+\varepsilon (x^s)^\top x^s,\\
    \text{which gives}&\\
    & 0\leq \ell(u^s_\varepsilon)-\ell^s \leq \varepsilon(x^s)^\top x^s \leq \varepsilon \max_{x \in \mathcal{X}} x^\top x. \nonumber
\end{split}
\end{equation}
The first inequality in the first line holds since $\varepsilon (x^s_\varepsilon)^\top x^s_\varepsilon \geq 0$. The second inequality in the first line is from optimality. If $R=\displaystyle \max_{x \in \mathcal{X}} x^\top x\leq 1$, then choosing $\varepsilon=\gamma$ satisfies \eqref{eq: modifed cost difference}. Otherwise, $\varepsilon=\frac{\gamma}{R}$ guarantees \eqref{eq: modifed cost difference}.  $\square$

The EMPC problem with the modified cost function is
\begin{subequations} \label{eq: Modified EMPC Formulation}
    \begin{align}
        \min_{u_{0|t},\ldots,u_{N-1|t}} &\sum_{k=0}^{N-1} \ell(u_{k|t})+\varepsilon x_{k|t}^\top x_{k|t}, \label{eq: Modified objective function in control problem}\\
        \text{subject to} \hspace{0.3 cm} &k=0,1,\ldots,N-1,\\
        & x_{0|t}=x, \label{eq: Modified initialization in EMPC}\\
        & x_{k+1|t}=f(x_{k|t},u_{k|t}), \label{eq: Modified Model dynamics in EMPC}\\
        & x_{k|t} \in \mathcal{X},\label{eq: Modified State constraints}\\
        &u_{k|t} \in \mathcal{U},\label{eq: Modified Input constraints}\\
        & x_{N|t}=x^s_\varepsilon, \label{eq: Modified Terminal constraint}
    \end{align}
\end{subequations}
Similar to \eqref{eq: admissible initial states}, denote by $\mathcal{X}_N^\varepsilon$ the admissible set of initial states $x$ such that the EMPC scheme in \eqref{eq: Modified EMPC Formulation} is feasible at time $t=0$, i.e.,
\begin{equation} \label{eq: Admissble pair modified}
\begin{split}
&\mathcal{Z}_{N}^\varepsilon=\big\{(x,\mathbf{u}) \in \mathcal{X}\times \mathcal{U}^N \hspace{0.1 cm}|\hspace{0.1 cm} \exists \hspace{0.1 cm} x_{k}\in\mathcal{X},\hspace{0.1 cm} k=1,\ldots,N-1\\ &\text{such that}\hspace{ 0.1 cm}x_{0}=x, \hspace{0.1 cm} x_{k+1}=f(x_{k},u_{k}),
\hspace{0.1 cm} x_{N}=x_\varepsilon^s\big\}, 
\end{split}
\end{equation}

\begin{equation} \label{eq: admissible initial states modified}
    \mathcal{X}_{N}^\varepsilon=\{x \in \mathcal{X} \hspace{0.1 cm}| \hspace{0.1 cm} \exists \mathbf{u} \hspace{0.1 cm} \text{such that} \hspace{0.1 cm} (x,\mathbf{u})\in\mathcal{Z}_{N}^\varepsilon\}.
\end{equation}

The following lemma guarantees that the difference between the optimal value of the cost in \eqref{eq: Modified objective function in control problem} and \eqref{eq: objective function in control problem} is less than $\gamma$.     

\begin{lemma} \label{lem: cost difference with modified gamma}
    Let Assumption \ref{Ass: Strict convexity of l(u)} hold. Denote the open-loop costs from the original EMPC problem in \eqref{eq: EMPC Formulation} and the modified EMPC in \eqref{eq: Modified EMPC Formulation} at time $t$ by $J^*_t$ and $J^*_{t,\varepsilon}$, respectively. If $\varepsilon$ in \eqref{eq: Modified objective function in control problem} is chosen as $\min\{\gamma,\frac{\gamma}{R}\}$ with $R=\max_{x \in \mathcal{X}} x^\top x$, then 
    \[0\leq J^*_{t,\varepsilon}-J^*_t \leq N \cdot\min\{\gamma,\frac{\gamma}{R}\}.\] 
\end{lemma}
\noindent \textbf{Proof}. Denote the solutions of \eqref{eq: EMPC Formulation} and \eqref{eq: Modified EMPC Formulation} by $[u^0_{0|t}, \hspace{0.1 cm} \ldots, \hspace{0.1 cm} u^0_{N-1|t}]^\top$ and $[u^0_{0|t,\varepsilon}, \hspace{0.1 cm} \ldots, \hspace{0.1 cm} u^0_{N-1|t,\varepsilon}]^\top$, respectively. From optimality, we have
\begin{equation}
\begin{split}
    \sum_{k=0}^{N-1} \ell(u^0_{k|t,\varepsilon})&\leq \sum_{k=0}^{N-1} \left(\ell(u^0_{k|t,\varepsilon})+\varepsilon (x^0_{k|t,\varepsilon})^\top x^0_{k|t,\varepsilon}\right)\\
    & \leq \sum_{k=0}^{N-1} \left(\ell(u^0_{k|t})+\varepsilon (x^0_{k|t})^\top x^0_{k|t}\right). \nonumber
\end{split}
\end{equation}
Therefore, we have 
\begin{equation}
\begin{split}
     \sum_{k=0}^{N-1} &\left(\ell(u^0_{k|t,\varepsilon})-\ell(u^0_{k|t})\right) \leq  \sum_{k=0}^{N-1} \varepsilon (x^0_{k|t})^\top x^0_{k|t}\\
    & \leq \sum_{k=0}^{N-1} \min\{\gamma,\frac{\gamma}{R}\}=N \cdot \min\{\gamma,\frac{\gamma}{R}\}. \hspace{0.5 cm} \square \nonumber
\end{split}
\end{equation}

\begin{remark}
    From Lemma \ref{lem: cost difference with modified gamma}, we can choose  $\gamma$ such that the open-loop costs are arbitrarily close. 
\end{remark}

Next, we prove that under a strict dissipativity assumption, $x^s_\varepsilon$ is an asymptotically stable equilibrium.
\begin{assumption} \label{Ass: Modified strict dissipativity}
    There exists a function $\Bar{\lambda}(x): \mathcal{X}\rightarrow \mathbb{R}$ and a positive definite $\Bar{\rho}(||x||_{x^s_{\varepsilon}}):\mathcal{X}\rightarrow \mathbb{R}_{\geq 0}$ such that
\begin{subequations} \label{eq: modified strict dissipativity def with respect xs}
\begin{align}
    & \hspace{-1.45 cm}\bar{L}(x,u)\geq \Bar{\rho}(||x||_{x^s_\varepsilon}), \label{eq: Lower bound for modified rotated stage cost}\\
    &\hspace{-2.5  cm}\text{where} \nonumber\\
    \Bar{L}(x,u)=&\bar{\lambda}(x)-\bar{\lambda}(f(x,u))+\ell(u)-\ell(u_\varepsilon^s) \nonumber\\
    &+ \varepsilon(x^\top x-(x^s_\varepsilon)^\top x^s_\varepsilon), \label{eq: modified rotated stage cost}
\end{align}
\end{subequations}
for all $(x,u) \in \mathcal{X}\times \mathcal{U}$.
\end{assumption}
It can be verified that replacing the stage cost in \eqref{eq: Modified objective function in control problem} by the rotated cost in \eqref{eq: modified rotated stage cost} will not change the solution of \eqref{eq: Modified EMPC Formulation}. Hence, a Lyapunov candidate function is defined based on the optimal rotated cost
\begin{equation} \label{eq: Modified Lyapunov function}
        \bar{V}^0(x_t)= \sum_{k=0}^{N-1} \bar{L}(x^0_{k|t,\varepsilon},u^0_{k|t,\varepsilon}). 
\end{equation}

We make use of the following lemma to find an upper bounding class-$\mathcal{K}$ function for the Lyapunov candidate function.

\begin{lemma}[\citep{amrit2011economic}] \label{lem: upper bound for Lyap}
Let a function $V(x)$ be defined on a closed set $\bar{\mathcal{X}}$. If $V(\cdot)$ is continuous at the origin and satisfies $V(0) = 0$, then there exists a class $\mathcal{K}$ function $\bar{\alpha}(\cdot)$ such that
\[
    V(x) \leq \bar{\alpha}(\lvert \lvert x \rvert \rvert), \quad \forall x \in \bar{\mathcal{X}} .
\]
\end{lemma}

\begin{theorem} \label{thm: AS for the modified}
   Given an $\varepsilon>0$, suppose that Assumption \ref{Ass: Modified strict dissipativity} holds. Then, for any initial condition $x_0 \in \mathcal{X}_N^\varepsilon$ in \eqref{eq: admissible initial states modified}, $x_\varepsilon^s$ is an asymptotically stable equilibrium of the closed-loop system \eqref{eq: Nonlinear Model dynamics} under the EMPC law \eqref{eq: Modified EMPC Formulation}.
\end{theorem}
\noindent \textbf{Proof.}   
Since $\bar{L}(x,u) \ge 0$ from \eqref{eq: Lower bound for modified rotated stage cost}, each term in the sum in \eqref{eq: Modified Lyapunov function} is nonnegative. Therefore, the sum is lower bounded by any of its terms, in particular the first one. $x^0_{0|t,\varepsilon}=x_t$ is the given initial state and $u^0_{0|t,\varepsilon}=u_{0|t}$ is the solution applied to the system. Therefore, the lower bound for the Lyapunov function can be obtained from \eqref{eq: Lower bound for modified rotated stage cost} as 
\[
 \bar{V}^0(x_t)= \sum_{k=0}^{N-1} \bar{L}(x^0_{k|t,\varepsilon},u^0_{k|t,\varepsilon}) \geq \bar{L}(x_t,u_{0|t})\geq\Bar{\rho}(||x||_{x^s_\varepsilon}).
\]
The upper bound of $\bar{V}^0(\cdot)$ can also be found using Lemma \ref{lem: upper bound for Lyap}, where the origin is shifted to $x^s_\varepsilon$. Hence,
\[
\bar{\rho}(||x_t||_{x^s_\varepsilon}) \leq \bar{V}^0(x_t) \leq \bar{\alpha}(||x_t||_{x^s_\varepsilon}), \hspace{0.2 cm} \forall x_t \in \mathcal{X}^\varepsilon_N.
\]
Consider the shifted sequences in \eqref{eq: Shifted solution} with the terminal state and input replaced by $x^s_\varepsilon$ and $u^s_\varepsilon$, respectively. Denote the cost associated with the shifted sequences by $\hat{V}_\varepsilon(\cdot)$. Under Assumption \ref{Ass: Modified strict dissipativity}, the descent property of the Lyapunov function can be established along the same lines as in \eqref{eq: descent property for Lyp-like in EMPC 1}.
\begin{equation} \label{eq: descent property for Lyp-like in EMPC}
\begin{split}
    &\bar{V}^0(x_{t+1})- \bar{V}^0(x_t)\leq \hat{V}_\varepsilon(x_{t+1})- \bar{V}^0(x_t)\\
    & =-L(x_t,u_{0|t}^0)\leq -\bar{\rho}(||x_t||_{x^s_\varepsilon}). \nonumber
\end{split}
\end{equation}
Therefore, from Definition \ref{def: Lyapunov function definition}, $\bar{V}^0(x)$ satisfies Lyapunov conditions and $x^s_\varepsilon$ is an asymptotically stable equilibrium point from Lemma \ref{lem: AS theorem from Khalil}. $\square$

A summary of EMPC schemes and the obtained properties is given by Fig. \ref{fig: Summary of EMPC schemes}. 
\begin{figure}[h]
    \centering
    \includegraphics[scale=0.7]{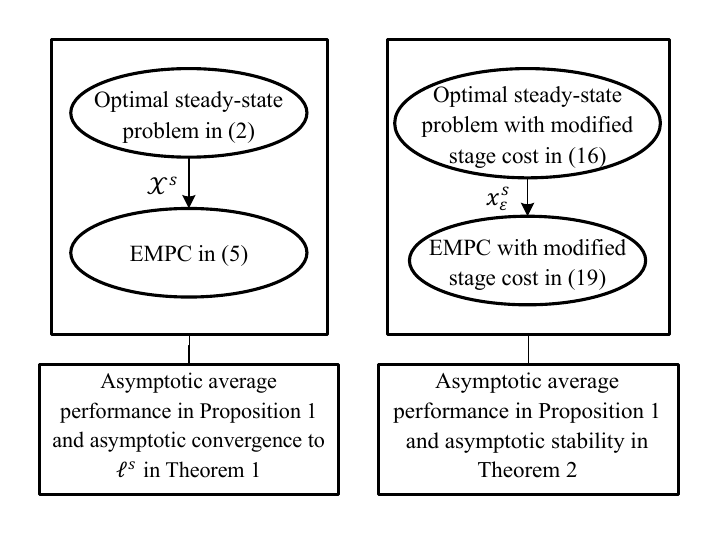}
    \caption{The proposed EMPC schemes and the obtained properties}
    \label{fig: Summary of EMPC schemes}
\end{figure}
\section{Extension to linear periodic systems} \label{sec: application to periodic}

In general, finding a storage function $\lambda(\cdot)$ to show strict dissipativity is not trivial. The storage function in \eqref{eq: strict dissipativity def with respect to Xs} should be chosen such that the rotated cost is positive definite with respect to all of the points in $\mathcal{Z}^s$, which makes finding the function more difficult than when $x^s$ is unique. Motivated by WDNs, where water demands and electricity prices are often periodic, we will in this section focus on linear periodic systems and show asymptotic convergence to the optimal steady state cost. The system dynamics for this class of systems is given by
\begin{equation} \label{eq: Model dynamics for linear periodic}
    x_{t+1}=Ax_t+Bu_t+B_dd_{t},
\end{equation}
where $d_t$ is a disturbance. 
\begin{assumption} \label{ass: periodic demands}
    The disturbance $d_t$ in \eqref{eq: Model dynamics for linear periodic} is deterministic and periodic with a known period $T>0$.
\end{assumption}

From Assumption \ref{ass: periodic demands}, the periodic disturbance can be represented by $d_{[t]_T}$, where $[t]_T$ denote the modulo operator which returns the remainder of $t$ divided by $T$. 

Under Assumption \ref{ass: periodic demands}, the system in \eqref{eq: Model dynamics for linear periodic} is lifted to a time-invariant system (Details are given in \ref{app: Lifted system}). Define the lifted variables\color{black}
\begin{equation}\label{eq:lifted_variables}
\begin{split}
    &\tilde{x}_{t+1} := \begin{bmatrix} &x_{tT+1}^\top & \cdots & x_{(t+1)T}^\top \end{bmatrix}^\top,\tilde{d} := \begin{bmatrix} &d_0^\top &\cdots &d_{T-1}^\top \end{bmatrix}^\top.\\
    & \tilde{u}_t := \begin{bmatrix} &u_{tT}^\top & \cdots & u_{(t+1)T-1}^\top \end{bmatrix}^\top,
\end{split}
\end{equation}
The lifted dynamics can be expressed as
\begin{equation}\label{eq:lifted_dynamics}
    \tilde{x}_{t+1} = \tilde{A}\tilde{x}_t + \tilde{B}\tilde{u}_t + \tilde{B}_d\,\tilde{d},
\end{equation}
where
\begin{align} \nonumber
    &\Tilde{A}=
    \begin{bmatrix}
        0 &0 &\ldots &A\\
        0 &0 & \ldots &A^2\\
        \vdots &\vdots & \vdots &\vdots\\
        0 &0 & \ldots &A^T
    \end{bmatrix}, \hspace{0.1 cm}
    \setlength\arraycolsep{1pt}
    \Tilde{B}=
    \begin{bmatrix}
B & 0 & \cdots & 0\\
A B & B & \cdots & 0\\
\vdots & \ddots & \ddots & \vdots\\
A^{T-1}B & \cdots & A B & B
\end{bmatrix},\\
    &\Tilde{B}_d=
\begin{bmatrix}
B_d & 0 & \cdots & 0\\
A B_d & B_d & \cdots & 0\\
\vdots & \ddots & \ddots & \vdots\\
A^{T-1}B_d & \cdots & A B_d & B_d
\end{bmatrix}.    \nonumber
\end{align}
In this section, we focus on a stage cost that depends on a given periodic external variable $\alpha_t$ with period $T$, i.e., $\ell(\alpha_{[t]_T},u_t)$.

\begin{remark}
In WDNs, the periodic external variable $\alpha_t$ is typically the price of electricity. Electricity tariffs usually have peak and off-peak periods during a day, and the price can be modelled as a periodic sequence with period $T=24$ hours.
\end{remark}

\begin{assumption} \label{ass: convexity of stage cost with alpha}
    Given an external variable $\alpha$, the stage cost function $\ell(\alpha,\cdot)$ is continuous and strictly convex. 
\end{assumption}
Under Assumption \ref{ass: periodic demands}, the optimal steady state problem can be defined as follows \citep{zanon2016periodic}:
\begin{subequations} \label{eq: OPSP equations}
    \begin{align}
        & \min_{\substack{x_{0|t_0},\ldots,x_{T-1|t_0}\\
        u_{0|t_0},\ldots,u_{T-1|t_0}}} \sum_{t=0}^{T-1} \ell(\alpha_{t|t_0},u_{t|t_0}), \label{eq: periodic Objective function}\\
        \text{subject to}& \hspace{0.2 cm} t=0,1,\ldots,T-2, \nonumber \\
        & x_{t+1|t_0}=Ax_{t|t_0}+B_uu_{t|t_0}+B_dd_{t|t_0}, \label{eq: periodic model dynamics}\\
        & x_{0|t_0}=Ax_{T-1|t_0}+B_uu_{T-1|t_0}+B_dd_{T-1|t_0},\label{eq: periodicity constraint}\\
        & x_{t|t_0} \in \mathcal{X}, \hspace{0.5 cm} u_{t|t_0} \in \mathcal{U} \label{eq: periodic state and input constraints},
    \end{align}
\end{subequations}
where \eqref{eq: periodicity constraint} ensures periodicity of the obtained state trajectory and $t_0$ is the starting time of the period. 
\begin{proposition}
The optimal cost in \eqref{eq: OPSP equations} is independent of the initial time $t_0$. Moreover, if an optimal solution at time $t_0$ is given by
\begin{equation} \label{eq: One solution to periodic SS problem}
    [x^{s\top}_0,\hspace{0.1 cm}\ldots,\hspace{0.1 cm}x^{s\top}_{T-1}]^\top, \hspace{0.2 cm}[u^{s\top}_0,\hspace{0.1 cm}\ldots,\hspace{0.1 cm}u^{s\top}_{T-1}]^\top.
\end{equation}
Then an optimal steady state trajectory at time $t_0+1$ is the shifted sequence
\begin{equation} \label{eq: shifted SS solution}
    [x^{s\top}_1,\hspace{0.1 cm}\ldots,\hspace{0.1 cm}x^{s\top}_{T-1}, \hspace{0.1 cm} x^{s\top}_0]^\top, \hspace{0.2 cm}[u^{s\top}_1,\hspace{0.1 cm}\ldots,\hspace{0.1 cm}u^{s\top}_{T-1}, \hspace{0.1 cm}u^{s\top}_{0}]^\top.
\end{equation}
\end{proposition}
\noindent \textbf{Proof}. Denote the value of the cost in \eqref{eq: periodic Objective function} using \eqref{eq: One solution to periodic SS problem} by $\ell_0$. A candidate feasible solution at time $t_0+1$ is obtained by shifting the sequences in \eqref{eq: One solution to periodic SS problem} giving the sequences in \eqref{eq: shifted SS solution}. The shifted sequences in \eqref{eq: shifted SS solution} meet all of the constraints in \eqref{eq: periodic model dynamics}-\eqref{eq: periodic state and input constraints}.

By \eqref{eq: One solution to periodic SS problem} and \eqref{eq: shifted SS solution}, the terms in \eqref{eq: periodic Objective function} are the same at times $t_0$ and $t_0+1$, but in a different order. Thus, the cost using \eqref{eq: shifted SS solution} equals $\ell_0$. Similarly, the shifted solutions with the corresponding cost of $\ell_0$ can be obtained for all $t_0+2, \ldots,t_0+T$. Hence, the optimal cost cannot increase with increasing initial time. Next, we prove that $\ell_0$ is the optimal cost for each $t_0=0,1,\ldots,T-1$ by contradiction.

Assume there exists a solution with a lower cost than $\ell_0$ for some $t_0 + i$, $i = 1,\ldots, T$. By assumption $\ell_i < \ell_0$, for $i = 1,\ldots, T$. By the periodicity of the stage cost, $\ell_T = \ell_0$. Thus, achieving a lower cost $\ell_T < \ell_0$, leads to a contradiction of the original assumption. Therefore, the cost $\ell_0$ is optimal, and the optimal cost is independent of the initial time $t_0$. Consequently, the shifted candidate solution in \eqref{eq: shifted SS solution} is optimal. \hspace{0.5cm}$\square$

By using the lifted system in \eqref{eq:lifted_dynamics}, the problem in \eqref{eq: OPSP equations} can be reformulated as a steady state problem in \eqref{eq: SS equations} as below:
\begin{subequations} \label{eq: Lifted OPSP}
    \begin{align}
        & \min_{\Tilde{x}_{t_0},\Tilde{u}_{t_0}} \Tilde{\ell}(\tilde{\alpha},\Tilde{u}_{t_0}), \\
        \text{{subject to}}&  \nonumber\\
        & \Tilde{x}_{t_0}=\Tilde{A}\Tilde{x}_{t_0}+\Tilde{B}_u\Tilde{u}_{t_0}+\Tilde{B}_d \Tilde{d} \label{eq: lifted model dynamics SS} \\
        & \Tilde{x}_{t_0} \in \tilde{\mathcal{X}}:=\underbrace{\mathcal{X}\times \ldots \times \mathcal{X}}_{\substack{\text{$T$ steps}}}, \label{eq: lifted SS state constraint} \\
        &\Tilde{u}_{t_0} \in \Tilde{\mathcal{U}}:=\underbrace{\mathcal{U}\times \ldots \times \mathcal{U}}_{\substack{\text{$T$ steps}}}
    \end{align}
\end{subequations}
where $\Tilde{\alpha}=[\alpha_0^\top, \ldots, \alpha_{T-1}^\top]^\top$. $\tilde{x}_{t_0}$ and $\tilde{u}_{t_0}$ are sequences in \eqref{eq:lifted_variables} at a given time $t_0$. 

We make the following assumption analogous to Assumption \ref{ass: existance of steady state solution} for the steady-state problem in \eqref{eq: Lifted OPSP}. 


\begin{assumption}\label{ass: slater condition}
Strong duality holds for the steady-state problem in \eqref{eq: Lifted OPSP}.
\end{assumption}

\begin{remark}
Assumption \ref{ass: slater condition} is relatively weak. The problem in \eqref{eq: Lifted OPSP} is convex, and strong duality holds when $\mathcal{X}$ and $\mathcal{U}$ are given by linear inequalities, such as box constraints. If nonlinear constraints are present, it is sufficient that there exists a feasible point where all nonlinear inequalities are strictly satisfied (Slater’s condition \cite{boyd2004convex}).
\end{remark}
\color{black}

Denote a solution to \eqref{eq: Lifted OPSP} by $\tilde{x}^s$ and $\tilde{u}^s$ with the optimal cost $\tilde{\ell}^s$. In view of Assumption \ref{ass: convexity of stage cost with alpha} and that the linear model is convex, $\tilde{u}^s$ is unique. The set of steady state solutions 
\begin{equation} \label{eq: lifted steady state trajectory}
\begin{aligned}
&\tilde{\mathcal{Z}}^s
= \Big\{
(\tilde{x},\tilde{u}^s) \in \tilde{\mathcal{X}} \times \tilde{\mathcal{U}}
\;|\;
\tilde{x} = \tilde{A}\tilde{x} + \tilde{B}_u \tilde{u}^s + \tilde{B}_d \tilde{d}, \\
& \hspace{3.8 cm} \tilde{\ell}(\tilde{\alpha}, \tilde{u}^s) = \tilde{\ell}^s
\Big\}, \\
&\tilde{\mathcal{X}}^s
= \left\{
\tilde{x} \in \tilde{\mathcal{X}}
\;\middle|\;
\exists \tilde{u} \in \tilde{\mathcal{U}} \text{ s.t. } (\tilde{x}, \tilde{u}) \in \tilde{\mathcal{Z}}^s
\right\},
\end{aligned}
\end{equation}
gives the optimal value of \eqref{eq: periodic Objective function}. Using $\tilde{\mathcal{X}}^s$ as the terminal constraint, a $T$-step EMPC problem with prediction horizon of $N=KT, K>1$ can be formulated as
\begin{subequations} \label{eq: T-step EMPC}
    \begin{align}
        & \min_{\Tilde{u}_{0|t},\ldots,\Tilde{u}_{K-1|t}} \sum_{i=0}^{K-1} \tilde{\ell}(\tilde{\alpha},\tilde{u}_{i|t}),\\
        \text{subject to} \hspace{0.3 cm} &i=0,1,\ldots,K-1, \nonumber\\
        & \Tilde{x}_{0|t}=\Tilde{x}, \label{eq: Lifted initialization in EMPC 3}\\
        & \Tilde{x}_{i+1|t}=\Tilde{A}\Tilde{x}_{i|t}+\Tilde{B}_{u}\Tilde{u}_{i|t}+\Tilde{B}_{d}\Tilde{d}, \label{eq: Lifted model dynamics in EMPC 3}\\
        & \Tilde{x}_{i|t} \in \mathcal{\tilde{X}}, \hspace{0.2 cm} \Tilde{u}_{i|t} \in \Tilde{\mathcal{U}},\label{eq: Lifted state constraints 2}\\
        & \Tilde{x}_{K|t} \in \Tilde{\mathcal{X}}^s.\label{eq: Lifted terminal constraint 3}
    \end{align}
\end{subequations}

The $T$-step EMPC problem has the same form as \eqref{eq: EMPC Formulation}. Thus, if the strict dissipativity assumption in Assumption \ref{Ass: Strict dissipativity} is met, the closed-loop cost obtained from \eqref{eq: T-step EMPC} converges to $\tilde{\ell}^s$ asymptotically by Theorem \ref{thm: Asymptotic convergence}. The following lemma and proposition provide the storage function for the system in \eqref{eq: Model dynamics for linear periodic}.

\begin{lemma} \label{lem:lower-bounding-k-infinity}
(\cite{amrit2011economic}). 
Let $\rho(x)$ be a positive definite function on a compact set 
$\mathcal{C}\subset\mathbb{R}^n$ containing the origin, i.e.,
\[
\rho(0)=0,\qquad \rho(x)>0,\ \ \forall x\in \mathcal{C}\setminus\{0\}.
\]
Then there exists a class $\mathcal{K}$ function $\gamma(\cdot)$ such that
\[
\rho(x)\ \ge\ \gamma(\|x\|), \qquad \forall x\in \mathcal{C}.
\]
\end{lemma}

\begin{proposition} \label{prop: strict dissipativity of lifted}
    The system in \eqref{eq: Model dynamics for linear periodic} is strictly dissipative on $\mathcal{X} \times \mathcal{U}$ with respect to the supply rate $\tilde{\ell}(\tilde{\alpha},\tilde{u})-\tilde{\ell}^s$ with $\tilde{\lambda}(\tilde{x})=\mu^\top\tilde{x}$ as storage function where $\mu$ is any optimal Lagrange multiplier associated with \eqref{eq: lifted model dynamics SS}.  
\end{proposition}

\noindent \textbf{Proof}. Under Assumption \ref{ass: slater condition}, strong duality holds. Let $\mu \geq 0$ be a Lagrange multiplier; hence,
\begin{equation}
\begin{split}
     \tilde{L}(\tilde{\alpha},\tilde{x},\tilde{u},\mu)&=\tilde{\ell}(\tilde{\alpha},\tilde{u})+\mu^\top (\tilde{x}-\tilde{A}\tilde{x}-\tilde{B}_u\tilde{u}-\tilde{B}_d\tilde{d})\\
    &-\tilde{\ell}^s\geq 0,
\end{split}
\end{equation}
for any feasible $\tilde{x}$ and $\tilde{u}$. Note that depending on the rank of $I-A$, \eqref{eq: lifted steady state trajectory} may admit an infinite number of feasible optimal steady state trajectories. The set of optimal Lagrange multipliers may not be unique and can also be infinite. Denote one optimal Lagrange multiplier by $\mu^*$. It is easy to check that $\tilde{L}(\tilde{\alpha},\tilde{x},\tilde{u},\mu^*)$ only eqauls zero if $(\tilde{x},\tilde{u})\in \tilde{\mathcal{Z}}^s$. Hence, from Lemma \ref{lem:lower-bounding-k-infinity}, there exists a $\beta \in \mathcal{K}$ such that
\[
\tilde{L}(\tilde{\alpha},\tilde{x},\tilde{u},\mu^*) \geq \beta(||\tilde{x}||_{\tilde{\mathcal{X}}^s}).\hspace{0.5 cm}\square
\]
\begin{corollary}\label{cor: T-step EMPC asymptotic convergence}
Let Assumptions \ref{ass: periodic demands} and \ref{ass: convexity of stage cost with alpha} hold. Consider the $T$-step EMPC problem in \eqref{eq: T-step EMPC}. Then, the closed-loop solution given by \eqref{eq: T-step EMPC} asymptotically converges to $\tilde{\mathcal Z}^s$. Consequently, $
\tilde{\ell}(\tilde{\alpha},\tilde{u}_{t}) \to \tilde{\ell}^s
$ as $t\to\infty$.
\end{corollary}

\noindent \textbf{Proof}. The lifted system in \eqref{eq:lifted_dynamics} and the T-step EMPC problem in \eqref{eq: T-step EMPC} have the same form as the EMPC problem in \eqref{eq: EMPC Formulation}. Moreover, by Proposition \ref{prop: strict dissipativity of lifted}, the lifted system is strictly dissipative with respect to the set $\tilde{\mathcal Z}^s$. Therefore, the conditions of Theorem \ref{thm: Asymptotic convergence} are satisfied for the lifted formulation. Hence, the closed-loop pair $(\tilde{x}_t,\tilde{u}_{t})$ converges asymptotically to $\tilde{\mathcal Z}^s$, and $\tilde{\ell}(\tilde{\alpha},\tilde{u}_{0|t}) \to \tilde{\ell}^s$. $\square$

\begin{corollary} \label{coro: asymptotic stability lifted}
    If the stage costs in \eqref{eq: Lifted OPSP} and \eqref{eq: T-step EMPC} are modified using \eqref{eq: Modified SS Objective function} and \eqref{eq: Modified objective function in control problem}, where $\varepsilon$ is chosen according to Lemma \ref{lem: bound for epsilon}, asymptotic stability with respect to a specific steady state is guaranteed. 
\end{corollary}
\noindent \textbf{Proof}. For the modified stage cost $\tilde{\ell}(\tilde{\alpha},\tilde{u})+\varepsilon\tilde{x}^\top\tilde{x}$, the optimal steady state pair ($\tilde{x}^s_\varepsilon, \tilde{u}^s_\varepsilon$) is unique. Similar to Proposition \ref{prop: strict dissipativity of lifted}, it can be verified that using the optimal Lagrange multiplier as the storage function, the closed-loop system is strictly dissipative with respect to $x^s_{\varepsilon}$. Hence, asymptotic stability can be guaranteed by Theorem \ref{thm: AS for the modified}. $\square$

\section{Case study: Richmond water distribution network} \label{sec: simulation results}
The Richmond water distribution network is part of the Yorkshire water supply area in the UK (\cite{van2004operational,exeter}). A schematic representation of the Richmond network is shown in Fig. \ref{fig: Richmond skeleton}.
\begin{figure}[b] 
      \centering
      \includegraphics[scale=0.055]{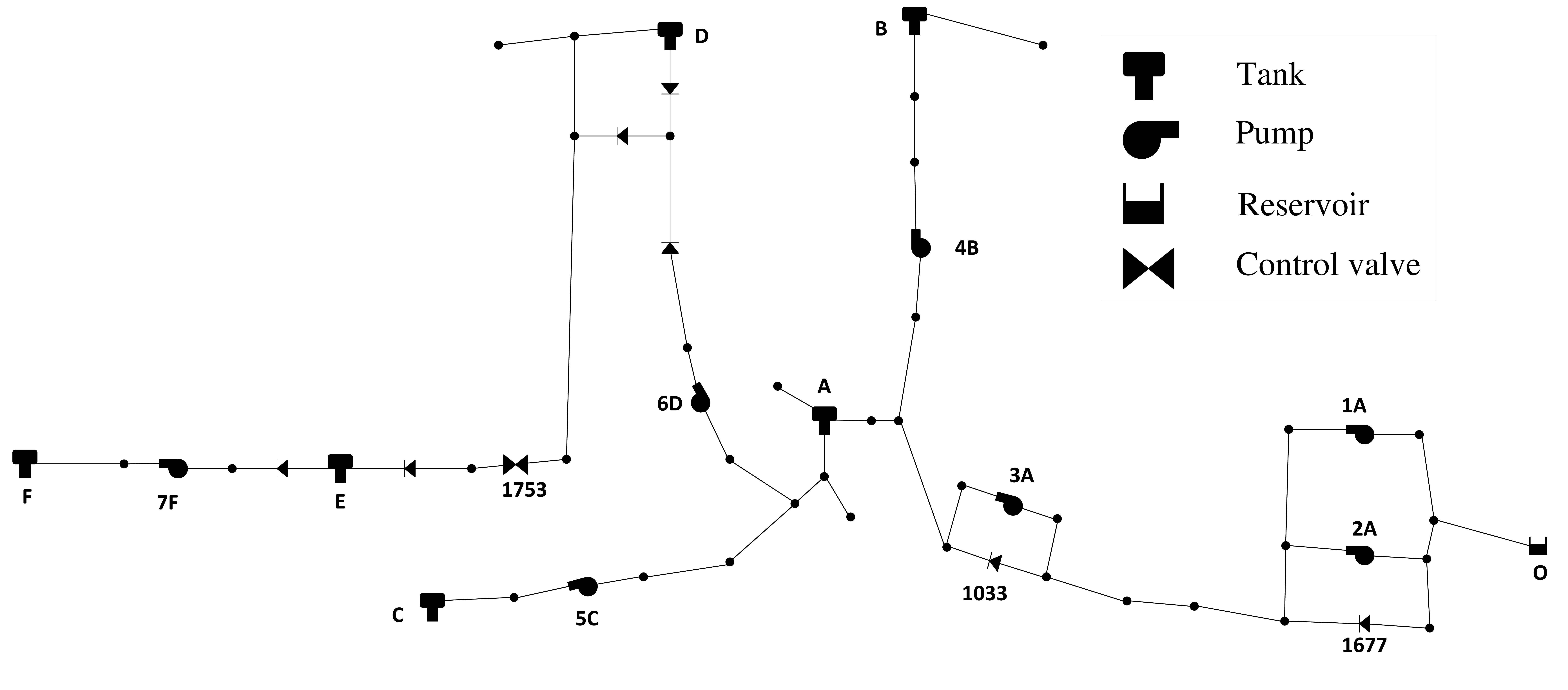}
      \caption{Richmond water distribution network}
      \label{fig: Richmond skeleton}
\end{figure}

The control-oriented model for this system can be formulated as (\cite{ocampo2012hierarchical,arastou2025optimization})
\begin{equation} \label{eq: dynamics for Richmond}
    x_{t+1}=x_t+B_u u_t+B_d d_{[t]_T},
\end{equation}
where $x \in \mathbb{R}^{6\times 1}$ are water levels in tanks, $u \in\mathbb{R}^{6\times 1}$ represents water flows through pumps, and $d \in \mathbb{R}^{10 \times 1 }$ is the vector of water demands. Moreover,
\begin{figure}[h]
\centering
\resizebox{\linewidth}{!}{$
\begin{gathered}
B_u =
\begin{bmatrix}
 \tfrac{\Delta t}{S_A} & -\tfrac{\Delta t}{S_A} & -\tfrac{\Delta t}{S_A} & -\tfrac{\Delta t}{S_A} & 0 & 0 \\
 0 & \tfrac{\Delta t}{S_B} & 0 & 0 & 0 & 0 \\
 0 & 0 & \tfrac{\Delta t}{S_C} & 0 & 0 & 0 \\
 0 & 0 & 0 & \tfrac{\Delta t}{S_D} & -\tfrac{\Delta t}{S_D} & 0 \\
 0 & 0 & 0 & 0 & \tfrac{\Delta t}{S_E} & -\tfrac{\Delta t}{S_E} \\
 0 & 0 & 0 & 0 & 0 & \tfrac{\Delta t}{S_F}
\end{bmatrix},
\\[1.2em]
B_d =
\begin{bmatrix}
 -\tfrac{\Delta t}{S_A} & -\tfrac{\Delta t}{S_A} & -\tfrac{\Delta t}{S_A} & 0 & 0 & 0 & 0 & 0 & 0 & 0 \\
 0 & 0 & 0 & -\tfrac{\Delta t}{S_B} & 0 & 0 & 0 & 0 & 0 & 0 \\
 0 & 0 & 0 & 0 & -\tfrac{\Delta t}{S_C} & 0 & 0 & 0 & 0 & 0 \\
 0 & 0 & 0 & 0 & 0 & -\tfrac{\Delta t}{S_D} & -\tfrac{\Delta t}{S_D} & -\tfrac{\Delta t}{S_D} & 0 & 0 \\
 0 & 0 & 0 & 0 & 0 & 0 & 0 & 0 & -\tfrac{\Delta t}{S_E} & 0 \\
 0 & 0 & 0 & 0 & 0 & 0 & 0 & 0 & 0 & -\tfrac{\Delta t}{S_F}
\end{bmatrix},
\end{gathered}
$}
\end{figure}

where the sampling time is chosen as $\Delta t = 1$ hour, and $S_i$ is the area of the $i$-th tank. Tank diameters can be found in \citep{exeter}. 

From the benchmark information, water demands are periodic with a period of $T = 24$ hours. The water demands are given by $d_t=m_t\bar{d}$, where $\bar{d} \in \mathbb{R}^{10 \times 1}$ and $m_t$ is the demand multiplier shown in Fig. \ref{fig: demand and electricity price Richmond}. The same demand multiplier was used for all of the demands. The demand multiplier has an average of 1. Each demand can have a different average.

The primary goal is to minimise pumping energy costs while satisfying demands. The cost function and constraints on states and inputs are as follows: 
\begin{subequations} \label{eq: cost function and constraints}
\begin{align}
&\ell(\alpha_{[k|t]_T},u_{k|t}) = \alpha_{[k|t]_T}^\top u_{k|t}+\Delta u_{k|t}^\top W\Delta u_{k|t}, \label{eq: obj function for Richmond}\\
& \underline{x}\leq x_{k|t} \leq \bar{x} , \hspace{0.2 cm} \underline{u}\leq u_{k|t} \leq \bar{u},  \label{eq: constraints in Richmond}
\end{align}
\end{subequations}
where, $\alpha_{[k|t]_T}=\bar{\alpha}_{[k|t]_T}\mathbf{1}_{6 \times 1}^\top$ is a vector of periodically varying electricity prices and $\Delta u_t=u_t-u_{t-1}$. $\bar{\alpha}_{[k|t]_T}$ is shown in Fig. \ref{fig: demand and electricity price Richmond}. The weighting matrix is set to $W=0.1I_{6}$. The first term in \eqref{eq: obj function for Richmond} represents the pumping energy costs, while the second term penalises input variations to prevent sudden changes in the control action. 

The stage cost in \eqref{eq: obj function for Richmond} depends on the previously applied input as $\Delta u_{0|t}=u_{0|t}-u^0_{t-1}$. Hence, the value of $u^0_{t-1}$ is initialised at the first time step in the control problem and then updated recursively. This formulation can be equivalently written as an augmented system, and the theoretical results for linear periodic systems in Proposition \ref{prop: strict dissipativity of lifted}, and Corollaries \ref{cor: T-step EMPC asymptotic convergence} and \ref{coro: asymptotic stability lifted} remain valid. The detailed reformulation is provided in \ref{app: reformulation of WDN}. 

The constraints in \eqref{eq: constraints in Richmond} limit water levels in tanks and admissible pump flows and are given by 
\begin{equation} \label{eq: Values of state and input constraints}
    \begin{split}
        &\bar{x}=[1.011,\hspace{0.1 cm} 1.095,\hspace{0.1 cm} 0.6,\hspace{0.1 cm} 0.633,\hspace{0.1 cm} 0.807,\hspace{0.1 cm}0.657]^\top,\hspace{0.2 cm} \underline{x}=-\bar{x},\\
        &\bar{u}=[50,50,50,50,50,50]^\top,\hspace{0.2 cm} \underline{u}=[0,0,0,0,0,0]^\top.
    \end{split}
\end{equation}
In the state constraints, $0$ means the middle of the tanks.
\begin{figure}[t]
    \centering
    \includegraphics[scale=0.7]{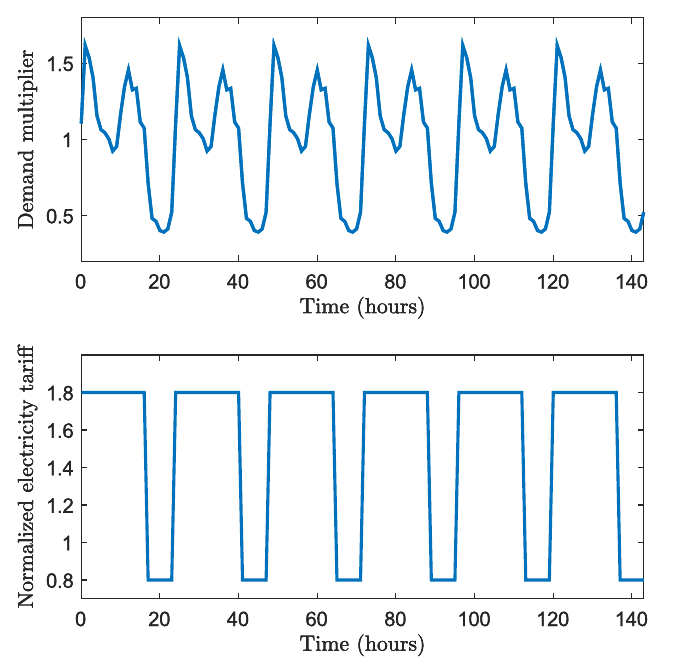}
    \caption{Demand multiplier and electricity tariff}
    \label{fig: demand and electricity price Richmond}
\end{figure}

\subsection{Optimal steady state problem}
The optimal steady state problem given by \eqref{eq: Lifted OPSP} is solved using the model, objective function, and constraints given by \eqref{eq: dynamics for Richmond},\eqref{eq: obj function for Richmond}, and \eqref {eq: constraints in Richmond}, respectively. The stage cost \eqref{eq: obj function for Richmond} is strictly convex with respect to $u$ and constraints are convex; hence, $[u_0^{s\top}, \ldots, u^{s\top}_{T-1}]^\top$ is unique. However, any state satisfying \eqref{eq: dynamics for Richmond} that remains feasible when applying $[u_0^{s\top}, \ldots, u^{s\top}_{T-1}]^\top$ is an optimal steady state trajectory. The solution for the water flow through pump $A$, and the optimal steady state for the water level in tank $A$ are shown in Figs. \ref{fig: us_A trajectory} and \ref{fig: xs_A manifold}, where the optimal steady state solution in Fig. \ref{fig: xs_A manifold} includes an infinite number of trajectories. Thus, the terminal set can be relaxed to the set of all feasible trajectories instead of a single trajectory.  

Next, asymptotic convergence of the closed-loop EMPC cost to the optimal steady state cost is shown using the $T$-step EMPC given in \eqref{eq: T-step EMPC}. 

\begin{figure}[t]
    \centering
    \includegraphics[scale=0.85]{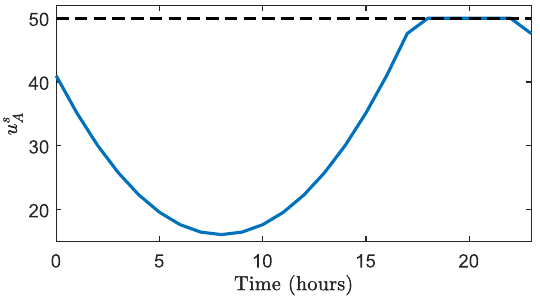}
    \caption{Optimal steady state trajectory for pump $A$}
    \label{fig: us_A trajectory}
\end{figure}
\begin{figure}[t]
    \centering
    \includegraphics[scale=0.65]{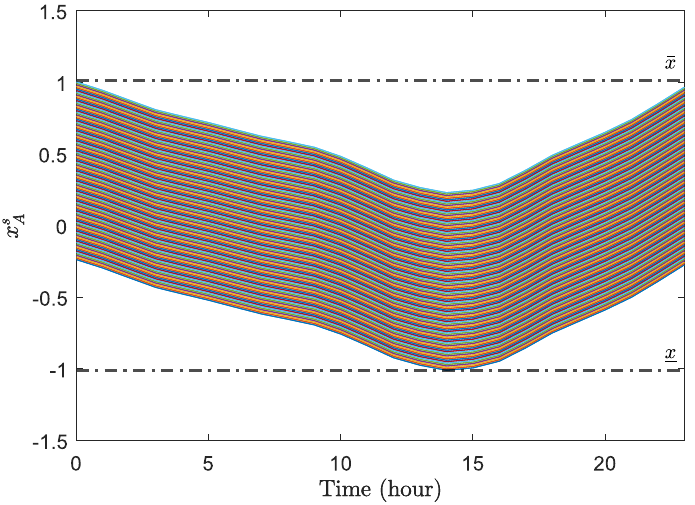}
    \caption{Some optimal steady-state trajectories of the water level in tank $A$, i.e., $x_A^s$. Dashed lines indicate the admissible bounds.}
    \label{fig: xs_A manifold}
\end{figure}

\subsection{Convergence of the closed-loop EMPC cost}

The $T$-step EMPC scheme in \eqref{eq: T-step EMPC} is applied, and the results are given below. The prediction horizon is set to $N=3T$, and the system is simulated for $6$ days ($144$ hours). Simulation results for tank $A$ and pump station $A$ are shown in Figs. \ref{fig: uA EMPC} and \ref{fig: xA EMPC}. The initial lifted state was set to the minimum value, i.e., each component of $\tilde{x}_0$ was chosen as $\underline{x}$. It can be seen from Fig. \ref{fig: xA EMPC} that the state trajectory converged to the minimum optimal steady state in Fig. \ref{fig: xs_A manifold}. This optimal steady state was reachable with less pumping than other trajectories in Fig. \ref{fig: xs_A manifold} in the transient period. The pumping is shown in Fig. \ref{fig: uA EMPC}. The control input reached the optimal input in steady state around time $t=20$ hours, and the closed-loop EMPC cost converged to the optimal cost as shown in Fig. \ref{fig: T-step EMPC cost}.

\begin{figure}[t]
    \centering
    \includegraphics[scale=0.64]{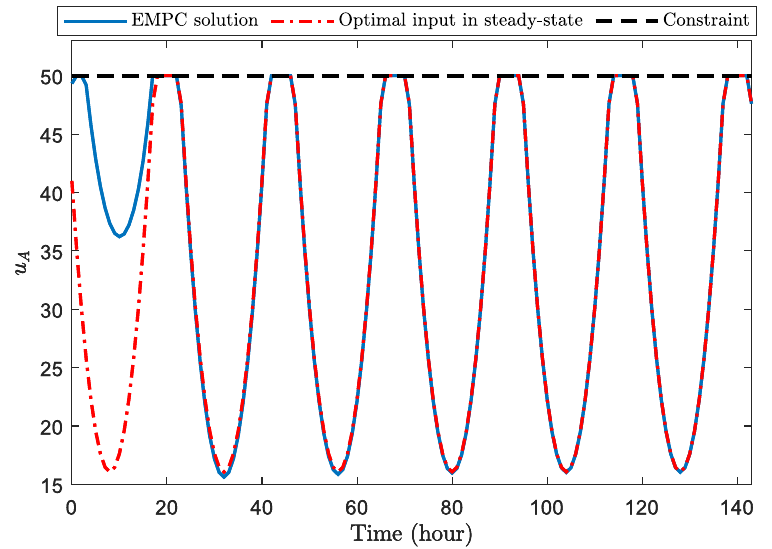}
    \caption{$u_A$ obtained by the $T$-step EMPC}
    \label{fig: uA EMPC}
\end{figure}
\begin{figure}[h]
    \centering
    \includegraphics[scale=0.7]{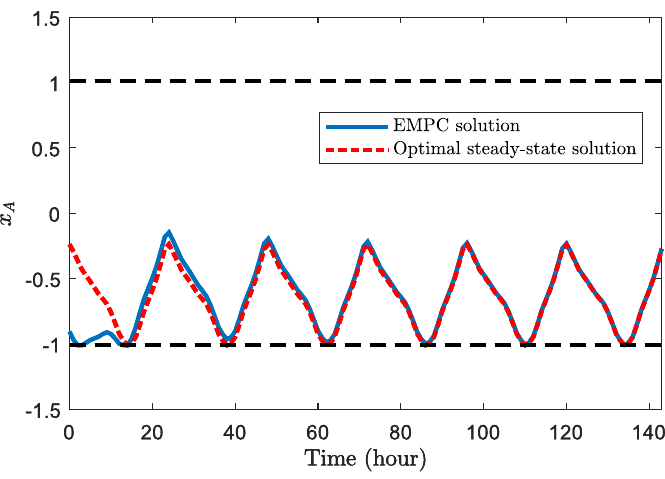}
    \caption{$x_A$ given by the $T$-step EMPC}
    \label{fig: xA EMPC}
\end{figure}
\begin{figure}[h]
    \centering
    \includegraphics[scale=0.73]{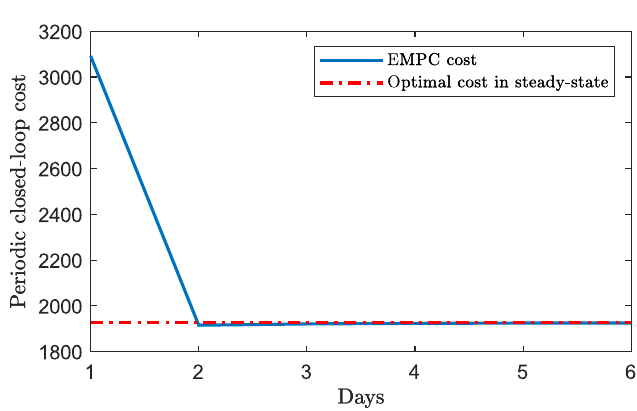}
    \caption{EMPC closed-loop cost at each period}
    \label{fig: T-step EMPC cost}
\end{figure} 
\subsection{Convergence using the modified stage cost}
As proposed in Section \ref{sec: asymptotic stability of modified xs}, the stage cost is modified with a term $\varepsilon \tilde{x}^\top \tilde{x}$ such that it is strictly convex with respect to both $\tilde{x}$ and $\tilde{u}$. 
The first step is to find $\varepsilon$ such that the optimal steady state cost and the open-loop EMPC cost are not changed significantly in comparison to those of the unmodified problems. 

From Lemma \ref{lem: bound for epsilon}, the upper bound on $\varepsilon$ can be computed. In this case study, the state constraint is given by \eqref{eq: constraints in Richmond} and \eqref{eq: Values of state and input constraints}. From \eqref{eq: Values of state and input constraints} and using the lifted state constraint in \eqref{eq: lifted SS state constraint}, $R=\displaystyle \max_{\tilde{x} \in \mathcal{\tilde{X}}} \tilde{x}^\top \tilde{x}=97.55$. Choosing $\varepsilon = 0.00102$ guarantees that the cost difference between optimal values of modified and unmodified steady state problems is less than $\gamma=0.1$. Denote the solutions of the modified steady state problem by $\tilde{x}^s_\varepsilon$ and $\tilde{u}_\varepsilon^s$. The unique solution of the optimal steady state problem for $x_{\varepsilon,A}^s$ and $u_{\varepsilon,A}^s$ is shown in Fig. \ref{fig: xs and us for the modified}. From Fig. \ref{fig: xs and us for the modified}, the optimal steady state is centred around $0$ due to the penalty on the states in the stage cost. \color{black} The optimal steady state cost using the modified stage cost in \eqref{eq: Modified SS equations} is $\bar{J}^s=1925.60$, while the optimal cost for the original problem in \eqref{eq: SS equations} was $J^s=1925.56$. This confirms that the cost difference is smaller than $\gamma=0.1$, as expected. 
\begin{figure}[t]
    \centering
    \includegraphics[scale=0.6]{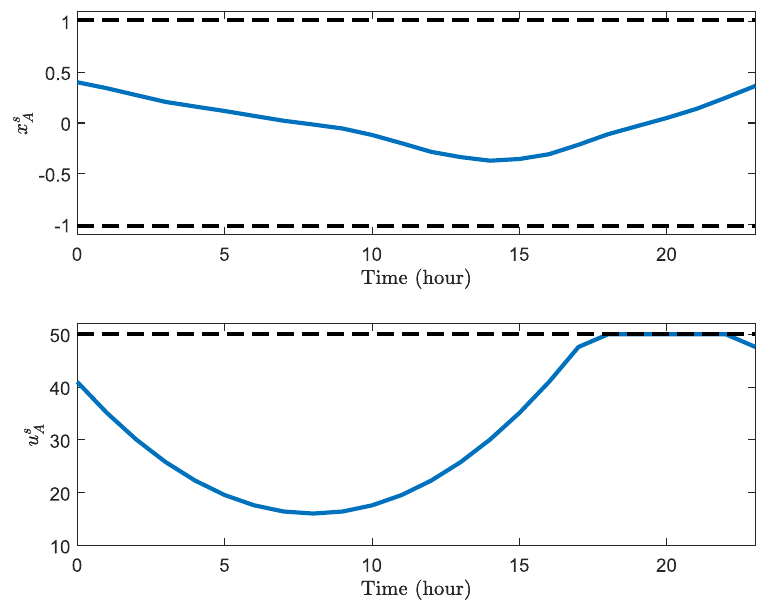}
    \caption{Optimal steady state solutions in the modified problem}
    \label{fig: xs and us for the modified}
\end{figure}

The EMPC in \eqref{eq: T-step EMPC} with a modified stage cost as in \eqref{eq: Modified objective function in control problem}, the same initial state  $\tilde{x}_0$ as in the unmodified problem in the previous section, and the terminal constraint $\tilde{x}_{K|t}=\tilde{x}_\varepsilon^s$ is used \color{black}, and the results for $x_A$ and $u_A$ are given in Figs. \ref{fig: xA and uA modified EMPC} and \ref{fig: EMPC cost of modified}. 
\begin{figure}[t]
    \centering
    \includegraphics[scale=0.65]{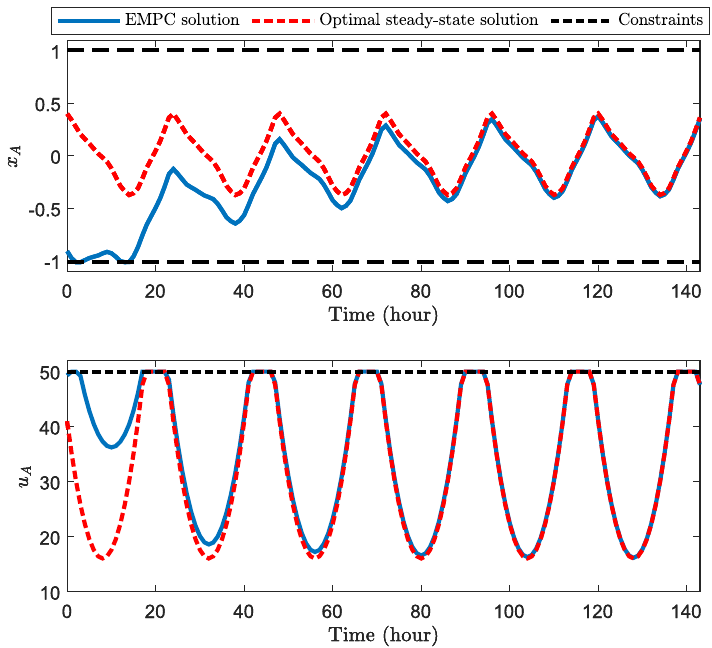}
    \caption{$x_A$ and $u_A$ obtained by the modified EMPC scheme}
    \label{fig: xA and uA modified EMPC}
\end{figure}
\begin{figure}[t]
    \centering
    \includegraphics[scale=0.65]{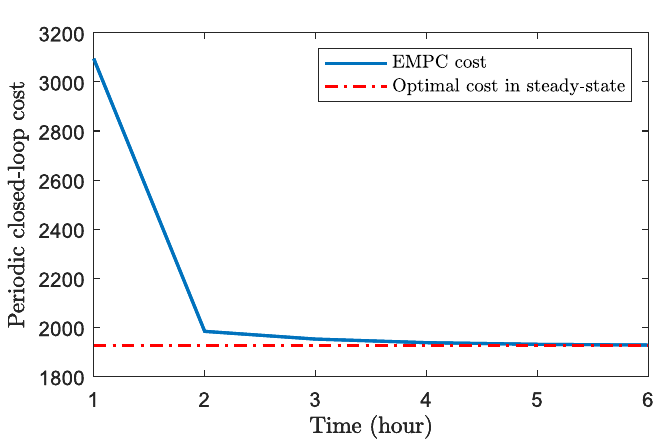}
    \caption{Periodic closed-loop EMPC cost using the modified stage cost}
    \label{fig: EMPC cost of modified}
\end{figure}

Comparing Figs. \ref{fig: uA EMPC}, \ref{fig: xA EMPC}, and \ref{fig: xA and uA modified EMPC}, more pumping was required during the transient period in Fig. \ref{fig: xA and uA modified EMPC} so that the state converges to the unique steady state trajectory centred around $0$. This is shown in Fig. \ref{fig: EMPC cost of modified} where the transient periodic cost is higher than in Fig. \ref{fig: T-step EMPC cost}. However, both costs converge to the corresponding optimal costs in steady state, which are almost the same because of the small value of $\gamma$. 
\color{black}

\section{Conclusion} \label{sec: conclusion}

EMPC schemes with purely economic objectives where the stage cost depends only on the control input have been analysed. The main contribution is that, under a strict dissipativity assumption with respect to the set of optimal steady states, it was shown that closed-loop trajectories asymptotically converge to the set of optimal steady states, thereby guaranteeing convergence of the economic cost. Since Lyapunov stability is not generally achieved in this setting, a modification of the stage cost was introduced that enabled asymptotic stability to a specific equilibrium while keeping the optimal steady state cost arbitrarily close to the cost in the original problem. The extension to linear periodic systems showed that the proposed framework naturally accommodates periodic costs and disturbances, with strict dissipativity established through strong duality. Application to a benchmark water distribution network validated the theoretical findings. 
\bibliographystyle{plainnat}       
\bibliography{References}           

@article{angeli2011average,
  title={On Average Performance and Stability of Economic Model Predictive Control},
  author={Angeli, David and Amrit, Rishi and Rawlings, James B},
  journal={IEEE Transactions on Automatic Control},
  volume={57},
  number={7},
  pages={1615--1626},
  year={2011},
  publisher={IEEE}
}

@book{rawlings2017model,
  title={Model Predictive Control: Theory, Computation, and Design},
  author={Rawlings, James Blake and Mayne, David Q and Diehl, Moritz and others},
  volume={2},
  year={2017},
  publisher={Nob Hill Publishing Madison, WI}
}

@book{Khalil:1173048,
      author        = "Khalil, Hassan K",
      title         = "{Nonlinear Systems; 3rd ed.}",
      publisher     = "Prentice-Hall",
      address       = "Upper Saddle River, NJ",
      year          = "2002",
}

@article{amrit2011economic,
  title={Economic Optimization Using Model Predictive Control with a Terminal Cost},
  author={Amrit, Rishi and Rawlings, James B and Angeli, David},
  journal={Annual Reviews in Control},
  volume={35},
  number={2},
  pages={178--186},
  year={2011},
  publisher={Elsevier}
}

@article{van2004operational,
  title={Operational Optimization of Water Distribution Systems Using a Hybrid Genetic Algorithm},
  author={Van Zyl, Jakobus E and Savic, Dragan A and Walters, Godfrey A},
  journal={Journal of Water Resources Planning and Management},
  volume={130},
  number={2},
  pages={160--170},
  year={2004},
  publisher={American Society of Civil Engineers}
}

@misc{exeter,
  author       = {{Centre for Water Systems, University of Exeter}},
  title        = {Richmond Skeleton Water Distribution Network},
  year         = {2001},
  note         = {Benchmark water distribution network, Yorkshire Water, UK},
  howpublished = {\url{https://www.exeter.ac.uk/research/centres/cws/resources/benchmarks/}}
}

@article{muller2016economic,
  title={Economic Model Predictive Control without Terminal Constraints for Optimal Periodic Behavior},
  author={M{\"u}ller, Matthias A and Gr{\"u}ne, Lars},
  journal={Automatica},
  volume={70},
  pages={128--139},
  year={2016},
  publisher={Elsevier}
}

@article{kohler2020periodic,
  title={Periodic Optimal Control of Nonlinear Constrained Systems Using Economic Model Predictive Control},
  author={K{\"o}hler, Johannes and M{\"u}ller, Matthias A and Allg{\"o}wer, Frank},
  journal={Journal of Process Control},
  volume={92},
  pages={185--201},
  year={2020},
  publisher={Elsevier}
}

@article{kohler2018periodic,
  title={On Periodic Dissipativity Notions in Economic Model Predictive Control},
  author={K{\"o}hler, Johannes and M{\"u}ller, Matthias A and Allg{\"o}wer, Frank},
  journal={IEEE Control Systems Letters},
  volume={2},
  number={3},
  pages={501--506},
  year={2018},
  publisher={IEEE}
}

@book{boyd2004convex,
  title={Convex Optimization},
  author={Boyd, Stephen and Vandenberghe, Lieven},
  year={2004},
  publisher={Cambridge University Press}
}

@article{arastou2025optimization,
  title={Optimization-Based Network Partitioning for Distributed and Decentralized Control},
  author={Arastou, Alireza and Wang, Ye and Weyer, Erik},
  journal={Journal of Process Control},
  volume={146},
  pages={103357},
  year={2025},
  publisher={Elsevier}
}

@article{ellis2017economic,
 title={Economic Model Predictive Control: Theory, Formulations and Chemical Process Applications},
  author={Ellis, Matthew and Liu, Jinfeng and Christofides, Panagiotis D.},
  year={2017},
  publisher={Springer},
}

@article{hu2023economic,
  title={Economic Model Predictive Control for Microgrid Optimization: A Review},
  author={Hu, Jiefeng and Shan, Yinghao and Yang, Yong and Parisio, Alessandra and Li, Yong and Amjady, Nima and Islam, Syed and Cheng, Ka Wai and Guerrero, Josep M and Rodr{\'\i}guez, Jos{\'e}},
  journal={IEEE Transactions on Smart Grid},
  volume={15},
  number={1},
  pages={472--484},
  year={2023},
  publisher={IEEE}
}

@inproceedings{wang2016periodic,
  title={Periodic Economic Model Predictive Control With Nonlinear-Constraint Relaxation for Water Distribution Networks},
  author={Wang, Ye and Alamo, Teodoro and Puig, Vicen{\c{c}} and Cembrano, Gabriela},
  booktitle={2016 IEEE Conference on Control Applications (CCA)},
  pages={1167--1172},
  year={2016},
  organization={IEEE}
}

@article{he2024dynamic,
  title={Dynamic Negotiation-Based Distributed {EMPC} With Varying Consensus Speeds of Heterogeneous Electric Vehicle Platoons},
  author={He, Defeng and Luo, Jie and Du, Haiping},
  journal={IEEE Transactions on Control Systems Technology},
  volume={32},
  number={4},
  pages={1495--1503},
  year={2024},
  publisher={IEEE}
}

@ARTICLE{Diehl2011,
  author={Diehl, Moritz and Amrit, Rishi and Rawlings, James B.},
  journal={IEEE Transactions on Automatic Control}, 
  title={A Lyapunov Function for Economic Optimizing Model Predictive Control}, 
  year={2011},
  volume={56},
  number={3},
  pages={703-707},
  doi={10.1109/TAC.2010.2101291}}

@ARTICLE{Grune_Dissipativity,
  author={Grüne, Lars},
  journal={IEEE Control Systems Magazine}, 
  title={Dissipativity and Optimal Control: Examining the Turnpike Phenomenon}, 
  year={2022},
  volume={42},
  number={2},
  pages={74-87},
  doi={10.1109/MCS.2021.3139724}}

@article{grune2016relation,
  title={On the Relation Between Strict Dissipativity and Turnpike Properties},
  author={Gr{\"u}ne, Lars and M{\"u}ller, Matthias A},
  journal={Systems \& control letters},
  volume={90},
  pages={45--53},
  year={2016},
  publisher={Elsevier}
}

@article{GRUNE20141187,
title={Asymptotic Stability and Transient Optimality of Economic {MPC} Without Terminal Conditions},
journal = {Journal of Process Control},
volume = {24},
number = {8},
pages = {1187-1196},
year = {2014},
note = {Economic nonlinear model predictive control},
issn = {0959-1524},
doi = {https://doi.org/10.1016/j.jprocont.2014.05.003},
author = {Lars Grüne and Marleen Stieler}
}

@article{MULLER_Average,
title={Convergence in Economic Model Predictive Control With Average Constraints},
journal = {Automatica},
volume = {50},
number = {12},
pages = {3100-3111},
year = {2014},
issn = {0005-1098},
doi = {https://doi.org/10.1016/j.automatica.2014.10.059},
author = {Matthias A. Müller and David Angeli and Frank Allgöwer and Rishi Amrit and James B. Rawlings}
}

@article{MULLER_Transient_Average,
title={Transient Average Constraints in Economic Model Predictive Control},
journal = {Automatica},
volume = {50},
number = {11},
pages = {2943-2950},
year = {2014},
issn = {0005-1098},
doi = {https://doi.org/10.1016/j.automatica.2014.10.024},
url = {https://www.sciencedirect.com/science/article/pii/S0005109814004051},
author = {Matthias A. Müller and David Angeli and Frank Allgöwer}
}

@article{zanon2022new,
  title={A New Dissipativity Condition for Asymptotic Stability of Discounted Economic {MPC}},
  author={Zanon, Mario and Gros, S{\'e}bastien},
  journal={Automatica},
  volume={141},
  pages={110287},
  year={2022},
  publisher={Elsevier}
}

@article{zanon2016periodic,
  title={Periodic Optimal Control, Dissipativity and {MPC}},
  author={Zanon, Mario and Gr{\"u}ne, Lars and Diehl, Moritz},
  journal={IEEE Transactions on Automatic Control},
  volume={62},
  number={6},
  pages={2943--2949},
  year={2016},
  publisher={IEEE}
}

@article{ocampo2012hierarchical,
  title={Hierarchical and Decentralised Model Predictive Control of Drinking Water Networks: Application to {Barcelona} Case Study},
  author={Ocampo-Martinez, Carlos and Barcelli, Davide and Puig, Vicen{\c{c}} and Bemporad, Alberto},
  journal={IET control theory \& applications},
  volume={6},
  number={1},
  pages={62--71},
  year={2012},
  publisher={IET}
}

\appendix
\renewcommand{\thesection}{Appendix \Alph{section}}
\renewcommand{\theequation}{\Alph{section}.\arabic{equation}}
\setcounter{equation}{0}
\section{Lifted Dynamical System} \label{app: Lifted system}

Iterating the dynamical system in \eqref{eq: Model dynamics for linear periodic} over $T$ steps yields
\begin{equation}
\begin{split}
x_{kT+i} =&A^{i}x_{kT}
+ \sum_{j=0}^{i-1} A^{i-1-j}B\,u_{kT+j} \\
&+ \sum_{j=0}^{i-1} A^{i-1-j}B_d\,d_j,
\quad i=1,\dots,T.
\end{split} \nonumber
\end{equation}
Stacking these $T$ state equations yields
\begin{equation} \nonumber
\scalebox{0.9}{$
\begin{aligned}
&\begin{bmatrix}
x_{kT+1}\\
x_{kT+2}\\
\vdots\\
x_{(k+1)T}
\end{bmatrix}
=
\underbrace{\begin{bmatrix}
0 & 0 & \cdots & 0 & A\\
0 & 0 & \cdots & 0 & A^2\\
\vdots &  & \ddots &\vdots  & \vdots\\
0 & 0 & \cdots & 0 & A^T
\end{bmatrix}}_{\tilde{A}}
\begin{bmatrix}
x_{(k-1)T+1}\\
x_{(k-1)T+2}\\
\vdots\\
x_{kT}
\end{bmatrix} \\[1em]
&+
\underbrace{\begin{bmatrix}
B & 0 & \cdots & 0\\
A B & B & \cdots & 0\\
\vdots & \ddots & \ddots & \vdots\\
A^{T-1}B & \cdots & A B & B
\end{bmatrix}}_{\tilde{B}_u}
\tilde{u}_{k}
+
\underbrace{\begin{bmatrix}
B_d & 0 & \cdots & 0\\
A B_d & B_d & \cdots & 0\\
\vdots & \ddots & \ddots & \vdots\\
A^{T-1}B_d & \cdots & A B_d & B_d
\end{bmatrix}}_{\tilde{B}_d}
\tilde{d}
\end{aligned}
$}
\end{equation}
It can be written as
\[
\tilde{x}_{k+1}
= \tilde{A}\tilde{x}_k
+ \tilde{B}_u\tilde{u}_k
+ \tilde{B}_d\tilde{d}.
\]
Since $\tilde{d}$ is periodic and constant over all periods, the lifted model is time-invariant.

\section{Reformulation of the Periodic EMPC Problem with an Input-Change Cost} \label{app: reformulation of WDN}

In practice, the economic objective may also include penalties for changes in the control input. This is common in WDNs, where large variations in pump flows are undesirable \citep{ocampo2012hierarchical,arastou2025optimization}. In these cases, the stage cost includes a term of the form 
\begin{equation} \label{eq: changes in the input}
    \Delta u_t = u_t - u_{t-1}.
\end{equation}
This introduces a dependence on the previously applied input;  hence, $u_{t-1}$ should be fed into the EMPC scheme at time $t$. The formulations in the main part of the paper depend only on the current decision variable. 

To address this, the objective can be rewritten by augmenting the system state with the previously applied input. With this representation, the EMPC problem has a structure similar to the periodic formulation in Section \ref{sec: application to periodic}, and the theories developed there can be applied in a similar way.


The stage cost is given by
\begin{equation} \label{eq: Cost WDN for reformulation}
\ell(\tilde{\alpha},\tilde{u}_{t},\tilde{u}_{t-1}) = \tilde{\alpha}^\top \tilde{u}_{t}+\Delta \tilde{u}_{t}^\top \tilde{W}\Delta \tilde{u}_{t}, \hspace{0.1 cm} \tilde{W}\succ 0
\end{equation}
For the lifted system given by \eqref{eq:lifted_dynamics}, the input difference term is given by
\begin{equation} \label{eq: delta_u_tilde_definition }
    \Delta \tilde{u}_{t}=
     \bar{M}\tilde{u}_t-\bar{N}u_{tT-1},
\end{equation}
where $\tilde{u}_{t}=\begin{bmatrix} 
    u_{tT}^\top &u_{tT+1}^\top & \cdots& u_{(t+1)T-1}^\top 
    \end{bmatrix}^\top$, and 
\begin{equation} \label{eq: definition of M and N bar}
\bar{M}_{mT \times mT}=
\begin{bmatrix}
I & 0 & 0 & \cdots & 0 \\
-I & I & 0 & \cdots & 0 \\
0 & -I & I & \cdots & 0 \\
\vdots & \ddots & \ddots & \ddots & \vdots \\
0 & \cdots & 0 & -I & I
\end{bmatrix}, \hspace{0.2 cm}
\bar{N}_{mT \times 1}=
\begin{bmatrix}
I  \\
0  \\
0  \\
\vdots \\
0 
\end{bmatrix},
\end{equation}
and $u_{tT-1}$ is the last element of the applied input sequence to the system, i.e., 
\begin{equation} \label{eq: utT-1 definition}
    u_{tT-1}=
     E\tilde{u}_{t-1},
\end{equation}
where $ E_{m \times mT}=  \begin{bmatrix}
        0 &0 &\ldots &0 &I
     \end{bmatrix}$.
Let $\hat{x}_t$ be $\tilde{x}_t$ augmented with $v_t=u_{tT-1}$
\begin{equation} 
    \hat{x}_t=\begin{bmatrix}
        \tilde{x}_t\\v_t
    \end{bmatrix},
\end{equation}
and state space representation 
\begin{equation} \label{eq: augmented model dynamics WDN}
    \hat{x}_{t+1}=\hat{A}\hat{x}_t+\hat{B}_u \tilde{u}_t+\hat{B}_d \tilde{d},
\end{equation}
where 
\begin{equation} \nonumber
    \hat{A}=\begin{bmatrix}
        \tilde{A} &0\\
        0 &0
    \end{bmatrix}, \hspace{0.1 cm}\hat{B}=
    \begin{bmatrix}
    \tilde{B}_u\\E    
    \end{bmatrix}, \hspace{0.1 cm} \hat{B}_d=
    \begin{bmatrix}
        \tilde{B}_d\\0
    \end{bmatrix},
\end{equation}

The stage cost in \eqref{eq: Cost WDN for reformulation} can be expressed as
\begin{equation} \label{eq: Reformulated stage cost with delta u}
    \hat{\ell}(\tilde{\alpha},\hat{x}_t,\tilde{u}_t)=\tilde{\alpha}^\top \tilde{u}_{t}+(\bar{M}\tilde{u}_t-\hat{N}\hat{x}_t)^\top \tilde{W} (\bar{M}\tilde{u}_t-\hat{N}\hat{x}_t),
\end{equation}
where 
\[
\hat{N}_{1 \times (n+m)T}=\begin{bmatrix}
    0_{1 \times nT} & \bar{N}^\top
\end{bmatrix}.
\]

The augmented state constraint is also given by
\begin{equation}\label{eq: augmented constraint set}
\hat{\mathcal X}
=
\left\{
\hat{x} =
\begin{bmatrix}
\tilde{x}^\top
&v^\top
\end{bmatrix} ^\top
\;\middle|\;
\tilde{x} \in \tilde{\mathcal X}, \; v \in \mathcal U
\right\}.
\end{equation}

Accordingly, the optimal periodic steady-state problem for the augmented system is given by
\begin{subequations}\label{eq: augmented steady state problem}
\begin{align}
\min_{\hat{x},\tilde{u}} \quad & \hat{\ell}(\tilde{\alpha},\hat{x},\tilde{u}) \label{eq: augmented steady state problem a}\\
\text{s.t.}\quad
& \hat{x}=\hat{A}\hat{x}+\hat{B}_u\tilde{u}+\hat{B}_d\tilde{d}, \label{eq: augmented steady state problem b}\\
& \hat{x}\in \hat{\mathcal X}, \quad \tilde{u}\in \tilde{\mathcal U}, \label{eq: augmented steady state problem c}
\end{align}
\end{subequations}
where $\hat{x}=[\tilde{x}^\top \ v^\top]^\top$ and $v=E\tilde{u}$ at steady state. Denote an optimal solution by $(\hat{x}^s,\tilde{u}^s)$ and the corresponding optimal value by $\hat{\ell}^s$. Define the optimal steady-state set of the augmented system by
\begin{equation}\label{eq: augmented steady state set}
\begin{aligned}
\hat{\mathcal Z}^s =
\Big\{
(\hat{x},\tilde{u}) \in \hat{\mathcal X} \times \tilde{\mathcal U}
\;\Big|\;
& \hat{x} = \hat{A}\hat{x} + \hat{B}_u \tilde{u} + \hat{B}_d \tilde{d}, \\
& \hat{\ell}(\tilde{\alpha}, \hat{x}, \tilde{u}) = \hat{\ell}^s
\Big\}
\end{aligned}
\end{equation}
and its projection onto the augmented state space by
\begin{equation}\label{eq: augmented Xs}
\hat{\mathcal X}^s=
\left\{
\hat{x}\in \hat{\mathcal X}
\;\middle|\;
\exists \tilde{u}\in \tilde{\mathcal U}\text{ such that }(\hat{x},\tilde{u})\in \hat{\mathcal Z}^s
\right\}.
\end{equation}

\begin{remark}
If Assumption \ref{ass: slater condition} holds for \eqref{eq: Lifted OPSP}, then it also holds for \eqref{eq: augmented steady state problem}. Any feasible point $(\tilde{x},\tilde{u})$ for \eqref{eq: Lifted OPSP} is a feasible point $(\hat{x},\tilde{u})$ for \eqref{eq: augmented steady state problem} by defining $\hat{x}=[\tilde{x}^\top,E\tilde{u}^\top]^\top$, and the augmentation adds equality constraints, which do not affect strong duality.
\end{remark}

The corresponding T-step EMPC problem for the augmented system is
\begin{subequations}\label{eq: augmented TEMPC}
\begin{align}
\min_{\tilde{u}_{0|t},\ldots,\tilde{u}_{K-1|t}} \quad &
\sum_{i=0}^{K-1}\hat{\ell}(\tilde{\alpha},\hat{x}_{i|t},\tilde{u}_{i|t}) \label{eq: augmented TEMPC a}\\
\text{subject to}\quad &i=0,\ldots,K-1\\
& \hat{x}_{0|t}=\hat{x}_t, \label{eq: augmented TEMPC b}\\
& \hat{x}_{i+1|t}=\hat{A}\hat{x}_{i|t}+\hat{B}_u\tilde{u}_{i|t}+\hat{B}_d\tilde{d}, \label{eq: augmented TEMPC c}\\
& \hat{x}_{i|t}\in \hat{\mathcal X}, \quad \tilde{u}_{i|t}\in \tilde{\mathcal U}, \label{eq: augmented TEMPC d}\\
& \hat{x}_{K|t}\in \hat{\mathcal X}^s. \label{eq: augmented TEMPC e}
\end{align}
\end{subequations}
The augmented initial state is given by
\begin{equation}\label{eq: augmented initial state repeated}
\hat{x}_t=
\begin{bmatrix}
\tilde{x}_t\\
E\tilde{u}_{t-1}
\end{bmatrix},
\end{equation}
where $E\tilde{u}_{t-1}=u_{tT-1}$ is the last applied input from the previous period. The set of admissible states is defined as 
\begin{equation} \label{eq: Region of attraction for augmented TEMPC}
\hat{\mathcal{X}}_K =\left\{ \hat{x}_0 \in \hat{\mathcal X} \;\middle|\; \eqref{eq: augmented TEMPC} \text{ is feasible at } \hat{x}_0
\right\}.
\end{equation}

\begin{proposition}\label{prop: augmented strict dissipativity}
Let Assumption \ref{ass: slater condition} hold. Then, the augmented system is strictly dissipative with the supply rate
$\hat{\ell}(\tilde{\alpha},\hat{x},\tilde{u})-\hat{\ell}^s$
with respect to the set $\hat{\mathcal Z}^s$.
\end{proposition}

\noindent\textbf{Proof.}
The stage cost in \eqref{eq: Reformulated stage cost with delta u} is convex since the first term
$\tilde{\alpha}^\top \tilde{u}$ is linear in $(\hat{x},\tilde{u})$ and the second term
$(\bar{M}\tilde{u}-\hat{N}\hat{x})^\top \tilde{W}(\bar{M}\tilde{u}-\hat{N}\hat{x})$ is a
composition of a convex function with an affine function.
The augmented steady-state problem in
\eqref{eq: augmented steady state problem} is therefore a convex optimisation problem and under Assumption \ref{ass: slater condition}, strong duality applies.

Let $\hat{\mu}^\star$ be an optimal Lagrange multiplier associated with \eqref{eq: augmented steady state problem b}, and define the storage function $\hat{\lambda}(\hat{x})=(\hat{\mu}^\star)^\top \hat{x}.$ Then, for any feasible $(\hat{x},\tilde{u})\in \hat{\mathcal X}\times \tilde{\mathcal U}$,
\begin{equation} \nonumber
\hat{\ell}(\tilde{\alpha},\hat{x},\tilde{u})-\hat{\ell}^s
+(\hat{\mu}^\star)^\top\bigl(\hat{x}-\hat{A}\hat{x}-\hat{B}_u\tilde{u}-\hat{B}_d\tilde{d}\bigr) \geq 0.
\end{equation}
The above inequality is equal to zero if and only if
$(\hat{x},\tilde{u})\in \hat{\mathcal Z}^s$. Hence, by the same argument as in Proposition \ref{prop: strict dissipativity of lifted}, there exists a positive definite function such that
\[
\hat{\lambda}(\hat{x})
-\hat{\lambda}(\hat{A}\hat{x}+\hat{B}_u\tilde{u}+\hat{B}_d\tilde{d})
+\hat{\ell}(\tilde{\alpha},\hat{x},\tilde{u})-\hat{\ell}^s
\ge
\hat{\rho}\bigl(\|(\hat{x},\tilde{u})\|_{\hat{\mathcal Z}^s}\bigr),
\]
\hfill $\square$

\begin{corollary}\label{cor: augmented feasibility convergence}
The augmented EMPC problem in \eqref{eq: augmented TEMPC} is recursively feasible. Moreover, the closed-loop solution converges asymptotically to $\hat{\mathcal Z}^s$. Consequently,
$\hat{\ell}(\tilde{\alpha},\hat{x}_t,\tilde{u}_{0|t}) \to \hat{\ell}^s$ as $t\to\infty.$
\end{corollary}

\noindent \textbf{Proof.} The recursive feasibility follows by the same shifted-sequence argument as in
Proposition~1, with $\hat{\mathcal X}^s$ as the terminal set. The asymptotic convergence
follows from Proposition~\ref{prop: augmented strict dissipativity} by the same argument as
in Corollary~\ref{cor: T-step EMPC asymptotic convergence}. $\square$

If asymptotic stability of a specific equilibrium is desired, the stage cost can be modified by adding a small positive term $\varepsilon \hat{x}^\top \hat{x}$, with $\varepsilon>0$. Similar to Section \ref{sec: asymptotic stability of modified xs}, $\varepsilon$ can be chosen small so that the difference between the optimal costs of the modified and unmodified steady-state problems is arbitrarily small. 

\begin{corollary}\label{cor: augmented point dissipativity}
Given $\varepsilon>0$, define the modified stage cost
\begin{equation} \label{eq: augmented modified stage cost}
    \hat{\ell}_\varepsilon(\tilde{\alpha},\hat{x},\tilde{u})
=
\hat{\ell}(\tilde{\alpha},\hat{x},\tilde{u})
+
\varepsilon \hat{x}^\top \hat{x}.
\end{equation}
Let $(\hat{x}_\varepsilon^s,\tilde{u}_\varepsilon^s)$ denote the unique optimal solution of the corresponding modified steady-state problem. Let Assumption \ref{ass: slater condition} hold. Then, the augmented system is strictly dissipative, i.e., 
\begin{subequations} \label{eq: modified strict dissipativity def with respect xs 2}
\begin{align}
    &\hat{L}(\tilde{\alpha},\hat{x},\tilde{u})\geq \hat{\rho}(||\hat{x}||_{\hat{x}^s_\varepsilon}) \label{eq: rotated augmented modified stage cost}\\
    \text{where} &\\
    &\hat{L}(\tilde{\alpha},\hat{x},\tilde{u})= \hat{\mu}^\top_\varepsilon(\hat{x}-\hat{A}\hat{x}-\hat{B}_u \tilde{u}-\hat{B}_d\tilde{d})+ \nonumber\\
    &\hspace{1.5 cm}\hat{\ell}_\varepsilon(\tilde{\alpha},\hat{x},\tilde{u})-\hat{\ell}^s_\varepsilon \nonumber
    + \varepsilon(\hat{x}^\top \hat{x}-(\hat{x}_\varepsilon^s)^\top \hat{x}^s_\varepsilon) , \label{eq: modified rotated stage cost}
\end{align}
\end{subequations}
for all $(x,u) \in \mathcal{X}\times \mathcal{U}$, where $\hat{\mu}_\varepsilon$ is the optimal Lagrange multiplier associated with the steady state constraint in the corresponding augmented modified steady state problem.  
\end{corollary}
\noindent\textbf{Proof.} The proof follows the same lines as Proposition \ref{prop: augmented strict dissipativity}. The modified stage cost $\hat{\ell}_\varepsilon(\tilde{\alpha},\hat{x},\tilde{u})$ is convex in $(\hat{x},\tilde{u})$, and the additional term $\varepsilon \hat{x}^\top \hat{x}$ makes the modified steady-state problem admit a unique optimal solution $(\hat{x}_\varepsilon^s,\tilde{u}_\varepsilon^s)$. By strong duality, an optimal Lagrange multiplier is used to define a storage function, and the same argument as in Proposition \ref{prop: augmented strict dissipativity} gives strict dissipativity with respect to $\hat{x}^s_\varepsilon$.
\hfill $\square$

\begin{theorem}\label{thm: augmented asymptotic stability}
Let Assumption \ref{ass: slater condition} hold. Consider the augmented EMPC problem in \eqref{eq: augmented TEMPC} with modified stage cost in \eqref{eq: augmented modified stage cost} and terminal constraint $\hat{x}_{K|t}=\hat{x}_\varepsilon^s$. Then, $\hat{x}_\varepsilon^s$ is an asymptotically stable equilibrium of the closed-loop augmented system with region of attraction in \eqref{eq: Region of attraction for augmented TEMPC}.
\end{theorem}

\noindent\textbf{Proof.}
The proof is similar to that of Theorem \ref{thm: AS for the modified}. Let the optimal solution at time $t$ be
\[
\{\hat{x}_{0|t}^0,\hat{x}_{1|t}^0,\ldots,\hat{x}_{K|t}^0=\hat{x}_\varepsilon^s\},
\qquad
\{\tilde{u}_{0|t}^0,\tilde{u}_{1|t}^0,\ldots,\tilde{u}_{K-1|t}^0\},
\]
with $\hat{x}_{0|t}^0=\hat{x}_t$. Define the Lyapunov candidate function as
\[
\hat{V}^0(\hat{x}_t)
=
\sum_{k=0}^{K-1}\hat{L}_\varepsilon(\hat{x}_{k|t}^0,\tilde{u}_{k|t}^0).
\]
Since $\hat{L}_\varepsilon(\hat{x},\tilde{u})\ge 0$ from \eqref{eq: rotated augmented modified stage cost}, each term in the sum is nonnegative, and hence, it can be written from Corollary \ref{cor: augmented point dissipativity} that
\[
\hat{V}^0(\hat{x}_t)
\ge
\hat{L}_\varepsilon(\hat{x}_t,\tilde{u}_{0|t}^0)
\ge
\hat{\rho}(\|\hat{x}_t\|_{\hat{x}_\varepsilon^s}).
\]
Moreover, $\hat{V}^0(\cdot)$ is continuous and $\hat{V}^0(\hat{x}_\varepsilon^s)=0$. Thus, there exists a class-$\mathcal K$ function $\hat{\alpha}$ such that
\[
\hat{V}^0(\hat{x}_t)\le \hat{\alpha}(\|\hat{x}_t\|_{\hat{x}_\varepsilon^s}),
\]
from Lemma \ref{lem: upper bound for Lyap}. 
Now consider the shifted sequences
\[
\{\hat{x}_{1|t}^0,\hat{x}_{2|t}^0,\ldots,\hat{x}_{K|t}^0,\hat{x}_\varepsilon^s\},
\qquad
\{\tilde{u}_{1|t}^0,\tilde{u}_{2|t}^0,\ldots,\tilde{u}_{K-1|t}^0,\tilde{u}_\varepsilon^s\}.
\]
Denote the associated cost by $\hat{V}(\hat{x}_{t+1})$. By optimality,
\[
\hat{V}^0(\hat{x}_{t+1})-\hat{V}^0(\hat{x}_t)
\le
\hat{V}(\hat{x}_{t+1})-\hat{V}^0(\hat{x}_t).
\]
Using the shifted sequences, we have
\[
\hat{V}^0(\hat{x}_{t+1})-\hat{V}^0(\hat{x}_t)
\le
-\hat{L}_\varepsilon(\hat{x}_t,\tilde{u}_{0|t}^0)
\le
-\hat{\rho}(\|\hat{x}_t\|_{\hat{x}_\varepsilon^s}).
\]
Therefore,
\[
\hat{\rho}(\|\hat{x}_t\|_{\hat{x}_\varepsilon^s})
\le
\hat{V}^0(\hat{x}_t)
\le
\hat{\alpha}(\|\hat{x}_t\|_{\hat{x}_\varepsilon^s}),
\]
and
\[
\hat{V}^0(\hat{x}_{t+1})-\hat{V}^0(\hat{x}_t)
\le
-\hat{\rho}(\|\hat{x}_t\|_{\hat{x}_\varepsilon^s}).
\]
Hence, $\hat{V}^0$ is a Lyapunov function for the closed-loop augmented system with region of attraction \eqref{eq: Region of attraction for augmented TEMPC}. Hence, $\hat{x}_\varepsilon^s$ is an asymptotically stable equilibrium.
\hfill $\square$

\begin{remark}
Asymptotic stability of the augmented state $\hat{x}=(\tilde{x},v)$ implies that $\tilde{x}_t \to \tilde{x}_\varepsilon^s$ and $v_t \to v_\varepsilon^s$. From the augmented dynamics, $v_{t+1} = E\tilde{u}_t$, and hence $E\tilde{u}_t \to v_\varepsilon^s = E\tilde{u}_\varepsilon^s$. Since the steady-state solution is unique in this case, $\tilde{u}_t \to \tilde{u}_\varepsilon^s$.
\end{remark}

\end{document}